\documentclass[11pt]{article}

\usepackage{paralist}
\usepackage[dvipsnames,table]{xcolor}
\usepackage[margin=3.0cm]{geometry} % Changes the margins of the paper.
\usepackage{graphicx} % Simple graphics.

\usepackage{amssymb, amsfonts, amsmath, amscd, amsthm} % Symbols, fonts, diagrams, and environments from AMS.
\usepackage{mathtools} % Symbols such as inclusion arrows etc.
\usepackage{titlesec} % Section and subsection fonts.
\usepackage{csquotes} % Quotation marks.
\usepackage[shortlabels]{enumitem} % Changes line spacing within lists.
\usepackage{tikz-cd} % Diagrams.
\usepackage{wrapfig} % To get text to flow around images.
\usepackage{epigraph} % Quotations at the beginning of chapters.
\usepackage{mathrsfs} % For \mathscr
\usepackage{mathtools}
\usepackage{soul}
\usepackage[normalem]{ulem}
\usepackage[misc]{ifsym}
\usepackage{multicol}
\usepackage{todonotes}
\usepackage{hyperref} % Hyperlinks.
\usepackage{comment}

\def\mc{\mathcal}
\def\H{\mc H}
\def\S{\mc S}
\def\RR{\mathbb R}
\def\NN{\mathbb N}

\newcommand{\lvl}[1]{{\color{violet}#1}}
\def\opw{\operatorname{op}^w}
\def\symw{\operatorname{sym}^w}

\hypersetup{verbose=false, colorlinks, allcolors=MidnightBlue, linkcolor=MidnightBlue, urlcolor=OliveGreen, breaklinks = true}

\numberwithin{equation}{section}

\titleformat*{\section}{\Large \scshape\center}
\titleformat*{\subsection}{\fontsize{14}{14} \sffamily}

\theoremstyle{plain}
\newtheorem{theorem}{Theorem}
\newtheorem*{theorem*}{Theorem}

\newtheorem{lemma}[theorem]{Lemma}
\newtheorem{proposition}[theorem]{Proposition}
\newtheorem{corollary}[theorem]{Corollary}

\theoremstyle{definition}

\newtheorem*{definition*}{Definition}
\newtheorem{example}[theorem]{Example}

\theoremstyle{remark}
\newtheorem{remark}{Remark}

\DeclareMathOperator{\tr}{tr}

\newcommand{\hookuparrow}{\mathrel{\rotatebox[origin=c]{90}{$\hookrightarrow$}}}

\begin{document}
\pagenumbering{gobble}
\title{\huge{Heisenberg-smooth operators from the phase space perspective}}
\author{Robert Fulsche and Lauritz van Luijk}
\date{}
\maketitle
\pagenumbering{arabic}

\begin{abstract}
    Cordes' characterization of Heisenberg-smooth operators bridges a gap between the theory of pseudo-differential operators and quantum harmonic analysis (QHA).
    We give a new proof of the result by using the phase space formalism of QHA.
    Our argument is flexible enough to generalize Cordes' result in several directions: (1) We can admit general quantization schemes, (2) allow for other phase space geometries, (3) obtain Schatten-class analogs of the result, and (4) are able to characterize precisely `Heisenberg-analytic' operators.
    For (3), we use QHA to derive Schatten versions of the Calder\'on-Vaillancourt theorem, which might be of independent interest.
\end{abstract}

\tableofcontents

\section{Introduction}

Pseudodifferential operators are a well-understood class of linear operators, which prominently appear at many places in mathematical analysis. Besides being useful tools in PDE problems, they also appear naturally from a harmonic analysis point of view. 
One of the many well-known results, where it becomes clear that pseudodifferential operators are natural objects to study, is Cordes' theorem on Heisenberg-smooth operators. Before recalling this theorem, we need to introduce some notation.

On the Hilbert space $\mathcal H = L^2(\mathbb R^d)$, we can consider the unitary Weyl operators: 
\begin{align}\label{eq:weyl_ops}
    W_{(x, \xi)}f(y) = e^{iy \cdot \xi - i \frac{x\cdot \xi}{2}} f(y-x).
\end{align}
Here, $(x, \xi) \in \mathbb R^{2d}$ is a point in phase space. For simplicity, we will usually abbreviate such points by $z = (x, \xi)$, $w = (y, \eta) \in \mathbb R^{2d}$. 
The Weyl operators form a strongly continuous projective representation of the phase space $\mathbb R^{2d}$, which satisfies $W_z^\ast = W_{-z}$ as well as the exponentiated form of the CCR relations, i.e., 
\begin{align}\label{eq:ccr}
    W_z W_w = e^{i\sigma(z, w)/2}W_{z+w},
\end{align}
where $\sigma$ is the usual symplectic form on $\mathbb R^{2d}$: $\sigma(z, w)= y\cdot \xi - x\cdot \eta$. For an operator $A \in \mathcal L(\mathcal H)$, we set the phase space shift of the operator by $z \in \mathbb R^{2d}$ by $\alpha_z(A) := W_z A W_{-z}$. 

Applying this shift of operators, for example in the theory of pseudodifferential operators, has of course a long history, see, e.g., \cite{Folland1989}. Based on this shift, R.\ Werner developed the framework of \emph{Quantum Harmonic Analysis} in \cite{werner84}, which stimulated fruitful research in the past few years \cite{Berge_Berge_Fulsche, Berge_Berge_Luef_Skrettingland2022, Fulsche2020, Fulsche2024,Halvdansson2022, keyl_kiukas_werner16, luef_eirik2018, Luef_Skrettingland2021, Sk20}. We will not give a thorough introduction to the topic of quantum harmonic analysis (we refer to \cite{luef_eirik2018} or \cite{Fulsche_Galke2023} for basic accounts on the matter) and will only mention the facts that we need whenever it will be appropriate for our discussion. 

Using the shift of an operator, one can define the phase space partial derivatives of the operator $A \in \mathcal L(\mathcal H)$ by
\begin{align*}
    \partial_j A &= \lim_{t \to 0} \frac{\alpha_{te_j}(A) - A}{t}, \quad j = 1, \dots, 2d,
\end{align*}
provided the limit exists in the Banach space $\mathcal L(\mathcal H)$. Here, $e_j$ are the standard basis vectors of $\mathbb R^{2d}$. Having these notions of partial derivatives, one defines iterated partial derivatives $\partial^{\alpha}A$, $\alpha \in \mathbb N_0^{2d}$, in the obvious manner. In studying such differentiable operators, one should first consider the continuous elements:
\begin{align*}
    C(\mathcal H) = C^0(\mathcal H) = \{ A \in \mathcal L(\mathcal H): ~z \mapsto \alpha_z(A) \text{ is } \| \cdot\|_{op}\text{-cont.}\}.
\end{align*}
Note that $C(\mathcal H)$ was denoted $\mathcal C_1$ in other works on quantum harmonic analysis. 
We change the notation here to avoid confusion with $C^1$.
Among the class of bounded linear operators, it plays the same role as the bounded uniformly continuous functions within $L^\infty(\mathbb R^{2d})$.
$C(\mathcal H)$ is the most important $C^\ast$ algebra to study from the point of view of \emph{correspondence theory}, which relates translation invariant function spaces on phase space with their noncommutative counterpart: translation invariant subspaces of $\mathcal L(\H)$.
Further, it has a rich and well-understood Fredholm theory \cite{Fulsche_Hagger2019, Hagger2021}.

It is now natural to consider the $k$-times differentiable elements: For $k \in \mathbb N_0$ let 
\begin{align*}
    C^k(\mathcal H) = \{ A \in \mathcal L(\mathcal H): ~\partial^\alpha A \text{ exists and } \partial^\alpha A \in C(\mathcal H) \text{ for every } \alpha \in \mathbb N_0^d, ~|\alpha| \leq k\}.
\end{align*}
Note that we asked for $\partial^\alpha A \in C(\mathcal H)$ instead of $\partial^\alpha A \in \mathcal L(\mathcal H)$, as it is the more natural class to work within when it comes to considerations of quantum harmonic analysis. From here, one can go on to the space $C^\infty(\mathcal H) = \cap_{k \in \mathbb N} C^k(\mathcal H)$. From the viewpoint of quantum harmonic analysis, phase space derivatives (of quadratic forms) have recently been used to study problems of essential self-adjointness \cite{ESA, Fulsche_vanLuijk2023} as well as operator Sobolev inequalities \cite{Lafleche2024}. Nevertheless, without having the formal framework of quantum harmonic analysis at hand, H.~O.~Cordes already anticipated these concepts, the class $C^\infty(\mathcal H)$ was seemingly first studied by him. He termed the elements of $C^\infty(\mathcal H)$ \emph{Heisenberg-smooth operators}. His main result on this class is the following, see \cite[Theorem 1.2]{Cordes1979} or \cite[Chapter 8]{Cordes1995}:
\begin{theorem*}
$C^\infty(\mathcal H)$ consists exactly of the class of pseudodifferential operators with symbols in $C_b^\infty(\mathbb R^{2d})$.
\end{theorem*}
In his proof, Cordes used a combination of harmonic analysis on the phase space and PDE methods. It is the purpose of this note to present a proof that relies only on harmonic analysis on the phase space, avoiding PDE techniques. 
To this end, we will make use of some aspects quantum harmonic analysis. In the end, it will turn out that Cordes' result is essentially a consequence of natural mapping properties of Fourier transforms as well as inclusion properties of certain function/operator spaces. 
In our opinion, this approach not only gives a new elegant proof of the results but also offers a structural explanation of why it holds true.
Indeed, the method is also flexible enough to provide the result in $\Phi$-quantization ($\Phi$ a group homomorphism of the state space) for operators on $L^2(\mathbb R^{d_1} \times \mathbb T^{d_2} \times \mathbb Z^{d_3})$, as is discussed afterward. Besides having certain geometric flexibility, we can also modify the notion of smoothness: We can also consider Schatten-class versions of Cordes' theorem. Proving this version of the theorem hinges on Schatten-class versions of the Calder\'{o}n-Vaillancourt estimates, as well as reverse Calder\'{o}n-Vaillancourt estimates. Finally, we will also be able to prove a version of Cordes' theorem for \emph{Heisenberg-analytic} operators.

We very briefly describe the structure of the paper: In Section \ref{sec:1}, we describe the necessary tools from quantum harmonic analysis, as well as some embedding results on function spaces, which are necessary to prove Cordes' result in Weyl quantization. In Section \ref{sec:2}, we will describe how to adapt the methods from Section \ref{sec:1} to obtain the same result for $\tau$-quantization. Section \ref{sec:3} will be where we sketch how to apply essentially the same method to obtain Cordes' result on $L^2(\mathbb R^{d_1} \times \mathbb T^{d_2} \times \mathbb Z^{d_3})$. In Section \ref{sec:5}, we will discuss the Calder\'{o}n-Vaillancourt theorem from the QHA perspective and apply versions of it to obtain Schatten-class versions of Cordes' theorem. Finally, in Section \ref{sec:analytic}, Heisenberg-analytic operators will be discussed.

\section{Cordes' result on Heisenberg-smooth operators and its proof in Weyl quantization}\label{sec:1}

Before turning towards the characterization of Heisenberg-smooth operators, let us briefly use the opportunity to discuss some of the basic properties of the operator spaces $C^k(\mathcal H)$ and $C^{\infty}(\mathcal H)$.

Endowed with the norm $\| A\|_{C^k} = \max_{|\alpha| \leq k} \| \partial^\alpha A\|_{op}$, $C^k(\mathcal H)$ clearly turns into a Banach space. The partial derivatives of operators have some properties resembling classical differential operators on functions.
\begin{lemma}
    Let $A, B \in C^1(\mathcal H)$.
    \begin{enumerate}[(1)]
        \item (Linearity) If $\lambda, \mu \in \mathbb C$, then $\partial_j (\lambda A  + \mu B) = \lambda \partial_j A + \mu \partial_j B$.
        \item (Product rule) $AB \in C^1(\mathcal H)$ and $\partial_j(AB) = (\partial_j A ) B + A (\partial_j B)$.
        \item (Derivative of inverse) If $B \in C^1(\mathcal H)$ is invertible, then $B^{-1} \in C^1(\mathcal H)$ with $\partial_j B^{-1} = -B^{-1}(\partial_j B)B^{-1}$.
        \item (Quotient rule) If $B$ is invertible, then $AB^{-1}$ and $B^{-1}A \in C^1(\mathcal H)$ with 
        \begin{align*}
            \partial_j (AB^{-1}) = (\partial_j A) B^{-1} - AB^{-1}(\partial_j B) B^{-1}
        \end{align*}
        and
        \begin{align*}
            \partial_j(B^{-1}A) = -B^{-1}(\partial_j B) B^{-1} A + B^{-1}(\partial_j A).
        \end{align*}
    \end{enumerate}
\end{lemma}

\begin{proof}
    (1) is clear. The proof of the product rule (2) follows analogously to the classical product rule for the derivative of functions. The quotient rule (4) follows immediately from the product rule and the derivative of inverse formula. Hence, only (3) needs to be addressed. Note that, if we assume that $B^{-1}$ is differentiable, we obtain from the product rule
    \begin{align*}
        0 = \partial_j (BB^{-1}) = (\partial_j B) B^{-1} + B (\partial_j B^{-1})
    \end{align*}
    such that $\partial_j B^{-1} = -B^{-1} (\partial_j B)B^{-1}$. 
    Thus, we have to verify that the differential quotient of $B^{-1}$ converges to this expression in norm. Using that $\alpha_z(B^{-1}) = \alpha_z(B)^{-1}$, we have:
    \begin{align*}
        &\left  \| \frac{\alpha_{te_j}(B^{-1}) - B^{-1}}{t} + B^{-1}\partial_j(B) B^{-1} \right \| \\
        &\hspace{2cm}\leq \| B^{-1}\|^2 \left \| \frac{B\alpha_{te_j}(B)^{-1} B - B}{t} + \partial_j(B) \right \|\\
        &\hspace{2cm}= \| B^{-1}\|^2 \left \| \frac{\alpha_{-te_j}(B) B^{-1} \alpha_{-te_j}(B) - \alpha_{-te_j}(B)}{t} + \alpha_{-te_j}(\partial_j B)\right \|\\
        &\hspace{2cm}= \| B^{-1}\|^2 \left \| \alpha_{-te_j}(B) B^{-1}\left [ \frac{ \alpha_{-te_j}(B) - B}{t}\right ] + \alpha_{-te_j}(\partial_j B) \right \|
    \end{align*}
    Now, we have (with convergence in operator norm): $\alpha_{te_j}(B) \to B$ such that $\alpha_{-te_j}(B)B^{-1} \to I$, $(\alpha_{-te_j}(B) - B)/t \to -\partial_j B$, $\alpha_{-te_j}(\partial_j B) \to \partial_j B$. Hence, the whole expression on the right-hand side converges to $0$.
\end{proof}

Of course, analogous formula\lvl e for the product and quotient rule carry over to higher order derivatives. We leave it to the interested reader to formulate the precise identities. An immediate consequence of the product and quotient rule is the following:
\begin{corollary} Let $k \in \mathbb N_0$.
    \begin{enumerate}
        \item $C^k(\mathcal H)$ is a Banach$^\ast$-algebra.
        \item $C^k(\mathcal H)$ is spectrally invariant in $\mathcal L(\mathcal H)$.
    \end{enumerate}
\end{corollary}
\begin{proof}
    (1) follows immediately from the product rule. 
    Regarding (2), we note that $C^0(\mathcal H)$ is spectrally invariant since it is a $C^\ast$ subalgebra of $\mathcal L(\mathcal H)$.
    For $k \geq 1$, $C^k(\mathcal H)$ is spectrally invariant as a subalgebra of $C^0(\mathcal H)$ by the quotient rule, hence by spectral invariance for $k = 0$ it is also spectrally invariant in $\mathcal L(\mathcal H)$. 
    \end{proof}
Now, one can consider the set of smooth elements:
\begin{align*}
    C^\infty(\mathcal H) = \{ A \in \mathcal L(\mathcal H): ~\partial^\alpha A \text{ exists for all } \alpha \in \mathbb N_0^{2d}\}.
\end{align*}
Note that we don't have to explicitly ask for the derivatives to be in $C^0(\mathcal H)$: Since the derivatives are again differentiable, they are automatically contained in $C^0(\mathcal H)$. From the properties of $C^k(\mathcal H)$ one obtains directly:
\begin{corollary}
    $C^\infty(\mathcal H)$ is a Fr\'{e}chet algebra with the family of seminorms $\| A\|_\alpha := \| \partial^\alpha A\|_{op}$, $\alpha\in\NN_0^{2d}$. It is a dense subalgebra of $C^0(\mathcal H)$ which is spectrally invariant in $\mathcal L(\mathcal H)$.
\end{corollary}
\begin{proof}
    The first claim is clear.
    Spectral invariance follows from the spectral invariance of the algebras $C^k(\mathcal H)$. Finally, density of $C^\infty(\mathcal H)$ in $C^0(\mathcal H)$ follows from standard considerations of quantum harmonic analysis: For $f \in L^1(\mathbb R^{2d})$ and $A \in \mathcal L(\mathcal H)$ we let $f \ast A := \int_{\mathbb R^{2d}} f(z) \alpha_z(A)~dz$. We refer to \cite{Fulsche2020,werner84} for standard properties of these convolution operators. 
    If $(f_n)_{n \in \mathbb N}$ is an approximate identity of $L^1(\mathbb R^{2d})$ and $A \in C^0(\mathcal H)$, then $f_n \ast A \to A$ in operator norm. If $f_n$ is, say, contained in the Schwartz functions on $\mathbb R^{2d}$ (e.g., it is a sequence of appropriate Gaussians), then $\partial^\alpha (f_n \ast A) = (\partial^\alpha f_n) \ast A$ such that $f_n \ast A \in C^\infty(\mathcal H)$.
\end{proof}
Note that, as an obvious consequence, $C^k(\mathcal H)$ is also dense in $C^0(\mathcal H)$.

Having already defined the shift of an operator, we define the shift of a function $f: \mathbb R^{2d} \to \mathbb C$ by 
\begin{align*}
    \alpha_z(f) = f( \ \cdot\, - z), \quad z \in \mathbb R^{2d}.
\end{align*}
For a function $f$ on $\mathbb R^{2d}$, we define its \emph{symplectic Fourier transform} by
\begin{align*}
    \mathcal F_\sigma(f)(w) = \frac{1}{(2\pi)^d}\int_{\mathbb R^{2d}}f(z) e^{i\sigma(z, w)}\,\mathrm{d}z.
\end{align*}
This expression is of course well-defined for $f \in L^1(\mathbb R^{2d})$, mapping $L^1(\mathbb R^{2d})$ to $C_0(\mathbb R^{2d})$ and $\mathcal S(\mathbb R^{2d})$ to $\mathcal S(\mathbb R^{2d})$. 
As is well-known, $\mathcal F_\sigma$ extends to an isometric isomorphism of $L^2(\mathbb R^{2d})$ and to a topological isomorphism of $\mathcal S'(\mathbb R^{2d})$. 
Further, the symplectic Fourier transform satisfies $\mathcal F_\sigma(\alpha_z(f)) = \gamma_z \mathcal F_\sigma(f)$, where $\gamma_z \lvl(g\lvl)$ is the \emph{modulation of $g$ by $z$}:
\begin{align*}
    \gamma_z(g)(w) = e^{i\sigma(z, w)}g(w).
\end{align*}
In general, i.e.\ for $f \in \mathcal S'(\mathbb R^{2d})$, these equalities have to be interpreted weakly. In \cite{Berge_Berge_Fulsche}, a convenient notion of the \emph{modulation} of an operator was introduced that will turn out useful in our discussion:
\begin{align*}
    \gamma_z(A) := W_{\frac{z}{2}} A W_{\frac{z}{2}}.
\end{align*}

For an operator $A \in \mathcal T^1(\mathcal H)$, where $\mathcal T^1(\mathcal H)$ denotes the trace class operators over $\mathcal H$, we define its Fourier-Weyl transform (also known as Fourier-Wigner transform) by
\begin{align*}
    \mathcal F_W(A)(w) = \tr(AW_w^\ast).
\end{align*}
This Fourier-Weyl transform is known to have a number of good properties, showing that it is a natural object to consider. We give some examples:
\begin{theorem}
\begin{enumerate}[(1)]
    \item $\mathcal F_W$ maps $\mathcal T^1(\mathcal H)$ injectively and continuously to $C_0(\mathbb R^{2d})$.
    \item $\mathcal F_W$ extends to a unitary map from $\mathcal T^2(\mathcal H)$, the Hilbert-Schmidt operators, to $L^2(\mathbb R^{2d})$.
    \item $\mathcal F_W$ yields a topological isomorphism from $\mathcal S(\mathcal H)$, the Schwartz operators, to $\mathcal S(\mathbb R^{2d})$. 
    \item $\mathcal F_W$ extends to a topological isomorphism from $\mathcal S'(\mathcal H)$, the tempered operators, to $\mathcal S'(\mathcal H)$.
    \item $\mathcal F_W$ satisfies $\mathcal F_W(\alpha_z(A)) = \gamma_z (\mathcal F_W(A))$ as well as $\mathcal F_W(\gamma_z(A)) = \alpha_z(\mathcal F_W(A))$.
\end{enumerate}
\end{theorem}
The space $\mathcal S(\mathcal H)$ of Schwartz operators was discussed in close detail in \cite{keyl_kiukas_werner16}.
It consists precisely of the continuous linear operators $\mathcal S'(\mathbb R^d) \to \mathcal S(\mathbb R^d)$ and $\mathcal S'(\mathcal H)$ consists of the continuous linear operators $\mathcal S(\mathbb R^d) \to \mathcal S'(\mathbb R^d)$. Equivalently, $\mathcal S(\mathcal H)$ consists of those operators with integral kernel in $\mathcal S(\mathbb R^{2d})$ and $\mathcal S'(\mathcal H)$ consists of operators with integral kernels in $\mathcal S'(\mathbb R^{2d})$.  
We have $\mathcal S(\mathcal H) \subset \mathcal T^1(\mathcal H)$ and $\mathcal L(\mathcal H) \subset \mathcal S'(\mathcal H)$. Both $\mathcal S(\mathcal H)$ and $\mathcal S'(\mathcal H)$ can be topologized in a natural way, cf.\ \cite{keyl_kiukas_werner16} for a detailed discussion. 

The Fourier-Weyl transform $\mathcal F_W$ also satisfies some forms of the convolution theorems as well as Bochner's theorem (cf.~\cite{werner84, luef_eirik2018, Fulsche_Galke2023}), which will nevertheless play no role here and are therefore omitted. Since $\mathcal F_W$ is an isomorphism on several levels, one may ask what the inverse map looks like. For appropriate functions, say $f \in L^1(\mathbb R^{2d})$, the inverse Fourier-Weyl transform $\mathcal F_W^{-1}(f)$ can simply be written as
\begin{align*}
\mathcal F_W^{-1}(f) = \int_{\mathbb R^{2d}} f(w) W_{w}~dw,
\end{align*}
interpreted as an integral in strong operator topology. Clearly, $\mathcal \|F_W^{-1}(f)\|_{op} \leq \| f\|_{L^1}$. At this point, we want to emphasize that the conventions of quantum harmonic analysis might be confusing to a person not used to them, simply for the reason that what we call $\mathcal F_W^{-1}$, the inverse Fourier-Weyl transform, is usually known as the group Fourier transform of the projective representation $W_z$. In quantum harmonic analysis, our conventions seem more favorable (even though this is clearly only a matter of taste), as $\mathcal F_W$ appears naturally in the Gelfand transform of the commutative Banach algebra $L^1(\mathbb R^{2d}) \oplus \mathcal T^1(\mathcal H)$, cf. \cite[Section 3.1]{Berge_Berge_Fulsche} for details.

Another interesting feature, which will be of crucial importance to this discussion, is the following:
\begin{lemma}\label{lem:WeylQuant}
We have $\mathcal F_W^{-1} \circ \mathcal F_\sigma = \mathrm{op}^w$, the Weyl quantization. 
\end{lemma}
This statement is certainly well-known. Nevertheless, the proof, which goes along explicit computations, is rarely explicitly given. Since we want to be as explicit as possible during this section of the paper, we give the proof for the reader's convenience.
\begin{proof}
We need to prove the equality $\mathcal F_W^{-1}(\mathcal F_\sigma(f)) = \mathrm{op}^w(f)$ only for $f \in \mathcal S(\mathbb R^{2d})$. By the continuity properties of the involved maps, it then extends to $\mathcal S'(\mathbb R^{2d})$. Let $\varphi \in \mathcal S(\mathbb R^d)$. The following computations have to be understood in the weak sense (i.e., they are valid upon pairing with another Schwartz function). For this one occasion, we will write $w = (x, \eta)$ and $z = (y, \xi)$. Now:
\begin{align*}
    \mathcal F_W^{-1}(\mathcal F_\sigma(f)) \varphi(t) &= \int_{\mathbb R^{2d}} \mathcal F_\sigma(f)(w) W_{w}(\varphi)(t)~dz\\
    &= \frac{1}{(2\pi)^d}\int_{\mathbb R^d} \int_{\mathbb R^d} \int_{\mathbb R^d} \int_{\mathbb R^d} f(y,\xi) e^{i\sigma((y, \xi), (x, \eta)} e^{i\eta \cdot t - i\frac{x \cdot \eta}{2}} \varphi(t-x)~dy~d\eta ~dx~d\xi
    \intertext{We substitute $t-x \mapsto x$ to obtain:}
    &= \frac{1}{(2\pi)^d} \int_{\mathbb R^d} \int_{\mathbb R^d} \varphi(x) e^{i\xi \cdot (t-x)} \int_{\mathbb R^d} f(y,\xi) \int_{\mathbb R^d} e^{i\eta (\frac{t+x}{2} - y)}~d\eta ~dy~dx~d\xi\\
    &= \int_{\mathbb R^d} \int_{\mathbb R^d}\varphi(x) e^{i\xi \cdot (t-x)} \int_{\mathbb R^d} f(y,\xi) ~d\delta_{\frac{t+x}{2}}(y)~dx~d\xi\\
    &= \int_{\mathbb R^d} \int_{\mathbb R^d} f(\frac{x+t}{2}, \xi) e^{i\xi \cdot (t-x)} \varphi(x)~dx~d\xi.
    \end{align*}
    This shows the equality in question.
\end{proof}
We obtain immediately:
\begin{corollary}
    $\alpha_z(\mathrm{op}^w(f)) = \mathrm{op}^w(\alpha_{z}(f))$ and $\gamma_\xi(\operatorname{op}^w(f)) = \operatorname{op}^w(\gamma_\xi(f))$.
\end{corollary}
Having discussed the covariance properties for the Weyl quantization (which are, of course, well-known), we now need to turn to function spaces. When not stated otherwise, proofs of the following facts can be found in the textbook \cite{Grochenig2001}. Fix $0 \neq g \in \mathcal S(\mathbb R^{2d})$. Then, for $f \in \mathcal S(\mathbb R^{2d})$, the short time Fourier transform (STFT) $V_g(f)$ is defined as
\begin{align*}
    V_g(f)(z, w) = \langle f, \gamma_w(\alpha_z(g))\rangle,
\end{align*}
where the pairing is the $L^2$ inner product. Here, $g$ is referred to as the \emph{window function} of the STFT. With the appropriate weak interpretation of this, one can define the STFT for any $f \in \mathcal S'(\mathbb R^{2d})$, given a window $0 \neq g \in \mathcal S(\mathbb R^{2d})$. The STFT is, within this framework, always a continuous function of $z$ and $w$. Note that we differ from the standard conventions of the STFT by a factor $2\pi$, i.e.\ within the usual definition, as it is used e.g.\ in \cite{Grochenig2001}, the STFT would agree with $V_g(f)(z, 2\pi w)$ in our notation. Nevertheless, this difference will cause no problems. When it comes to compatibility with the phase space operations on operators, our convention seems to be favorable.

Using the STFT, one can then go on to define \emph{modulation spaces}. There is a whole zoo of these spaces, and we refer to \cite{Grochenig2001} for a discussion on them. When defining modulation spaces, one fixes a window function $0 \neq g \in \mathcal S(\mathbb R^{2d})$. For our purposes, we only need the following space:
\begin{align*}
    M^{\infty, 1}(\mathbb R^{2d}) := \{ f \in \mathcal S'(\mathbb R^{2d}): V_g(f) \in L^{\infty, 1}(\mathbb R^{4d})\}.
\end{align*}
Here, the mixed norm Lebesgue space $L^{\infty, 1}(\mathbb R^{4d})$ is defined through the following norm: For measurable $h: \mathbb R^{4d} \to \mathbb C$ it is
\begin{align*}
    \| h\|_{L^{\infty, 1}} &= \int_{\mathbb R^{2d}} \|h(\,\cdot\,,w)\|_\infty~dw.
\end{align*}
The space $M^{\infty, 1}(\mathbb R^{2d})$ is now normed by $\| f\|_{M^{\infty, 1}, g} = \| V_g(f)\|_{L^{\infty, 1}}$. One can show that the space is independent of the choice of the window function $g$, and two different windows give rise to equivalent norms. 

The same game of STFTs can now be played with operators---see  \cite{Berge_Berge_Fulsche} or \cite{Dorfler_Luef_McNulty_Skrettingland2024}, where different versions of operator STFTs were already discussed: Fixing a window operator $0 \neq B \in \mathcal S(\mathcal H)$, we define the STFT of $A \in \mathcal S'(\mathcal H)$ as the function
\begin{align*}
    V_B A(z, w) = \langle A, \gamma_w (\alpha_z(B))\rangle.
\end{align*}
Here, the pairing is the natural extension of the Hilbert-Schmidt inner product. We now define
\begin{align*}
    M^{\infty, 1}(\mathcal H) := \{ A  \in \mathcal S'(\mathcal H): V_B A \in L^{\infty, 1}(\mathbb R^{4d})\}.
\end{align*}
Using Plancherel's theorem for $\mathcal F_W$ (resp.\ its extension to the pairing between $\mathcal S(\mathcal H)$ and $\mathcal S'(\mathcal H)$, we see that
\begin{align*}
    \langle A, \gamma_w(\alpha_z(B))\rangle &= \langle \mathcal F_W(A), \mathcal F_W(\gamma_w(\alpha_z(B)))\rangle\\
    &= \langle \mathcal F_W(A), \alpha_w(\gamma_z(\mathcal F_W(B)))\rangle\\
    &= \langle \mathcal F_\sigma \mathcal F_W(A), \gamma_w \alpha_z(\mathcal F_\sigma \mathcal F_W(B))\rangle.
\end{align*}
Since $0 \neq \mathcal F_\sigma ( \mathcal F_W(B)) \in \mathcal S(\mathbb R^{2d})$, we see that $A \in M^{\infty, 1}(\mathcal H)$ if and only if $\mathcal F_\sigma(\mathcal F_W(A)) \in M^{\infty, 1}(\mathbb R^{2d})$. If we norm $M^{\infty, 1}(\mathcal H)$ by $\| A\|_{M^{\infty, 1}, B} = \| V_BA\|_{L^{\infty, 1}}$, the above calculation shows that different choices of $B$ yield equivalent norms. We obtain that $f \mapsto \mathrm{op}^w(f)$ is an isomorphism from $M^{\infty, 1}(\mathbb R^{2d})$ to $M^{\infty, 1}(\mathcal H)$. If we endow both spaces with compatible norms (e.g.\ fix $0 \neq B \in \mathcal S(\mathcal H)$ and $g = \mathcal F_\sigma \mathcal F_W(B)$), then $f \mapsto \mathrm{op}^w(f)$ is even isometrically isomorphic from $M^{\infty, 1}(\mathbb R^{2d})$ to $M^{\infty, 1}(\mathcal H)$.

As the last step, we need to discuss certain embedding results. The first fact we need is:
\begin{lemma}[{\cite[Prop.\ 1.7(3)]{Toft2004}}]\label{lem:embMF}
    $M^{\infty, 1}(\mathbb R^{2d}) \hookrightarrow L^\infty(\mathbb R^{2d})$.
\end{lemma}
As is well-known, Weyl pseudodifferential operators with symbols contained in $M^{\infty, 1}(\mathbb R^{2d})$ are bounded in operator norm on $L^2(\mathbb R^d)$, $\| \mathrm{op}^w(f)\|_{op} \lesssim \| f\|_{M^{\infty, 1}, g}$, see \cite[Thm.\ 14.5.2]{Grochenig2001} for a textbook reference or \cite[Sec.\ 3]{Sjostrand1994} for the initial proof. In light of the isomorphism $f \mapsto \mathrm{op}^w(f)$ from $M^{\infty, 1}(\mathbb R^{2n})$ to $M^{\infty, 1}(\mathcal H)$ this proves:
\begin{lemma}\label{lem:embMO}
    $M^{\infty, 1}(\mathcal H) \hookrightarrow \mathcal L(\mathcal H)$.
\end{lemma}
Let us denote by $C_b^{2d+1}(\mathbb R^{2d})$ the bounded continuous functions on $\mathbb R^{2d}$ with $2d+1$ continuous and bounded derivatives. Then, we have:
\begin{lemma}[{\cite[Thm.\ 14.5.3]{Grochenig2001}}]\label{lem:embCF}
    $C_b^{2d+1}(\mathbb R^{2d}) \hookrightarrow M^{\infty, 1}(\mathbb R^{2d})$.
\end{lemma}
As Lemma \ref{lem:embMO} is the operator analogue of Lemma \ref{lem:embMF}, we now have to discuss the operator analogue of Lemma \ref{lem:embCF}, which is the following:
\begin{lemma}\label{lem:OCM}
    $C^{2d+1}(\mathcal H) \hookrightarrow M^{\infty, 1}(\mathcal H)$.
\end{lemma}
\begin{proof}
    We begin by noting that, for $A \in M^{\infty, 1}(\mathcal H)$ and $B \in \mathcal S(\mathcal H)$, we have, using that the Hilbert-Schmidt inner product is antilinear in the second entry,
\begin{align*}
    V_BA (z, w) = \langle A, \gamma_w \alpha_z(B)\rangle &= \theta_{z,w}\langle A, W_{z-\frac{w}{2}} B W_{-z-\frac{w}{2}}\rangle\\
    &= \theta_{z,w}\tr(A (W_{z-\frac{w}{2}} B W_{-z-\frac{w}{2}})^\ast)\\
    &= \theta_{z,w} \tr(A W_{z+\frac{w}{2}} B^\ast W_{-z + \frac{w}{2}}).
\end{align*}
Here, $\theta_{z,w}$ is some unimodular constant, depending on $z$ and $w$, as obtained from applications of the CCR equation \eqref{eq:ccr}. Further, we have
\begin{align*}
    \mathcal F_W(A \alpha_z(B^\ast))(w) &= \tr(A W_z B^\ast W_{-z} W_w)\\
    &= \theta_{z,w}'\tr(A W_z B^\ast W_{-z+w})\\
    &= \theta_{z,w}' \tr(A W_{z' + \frac{w}{2}} B^\ast W_{-z' + \frac{w}{2}}),
\end{align*}
where we substituted $z \mapsto z' + \frac{w}{2}$ and $\theta_{z,w}'$ is again some unimodular constant. This shows that
\begin{align*}
    |V_B A (z, w)| =  |\mathcal F_W(A \alpha_{z-\frac{w}{2}}(B^\ast))(w)|.
\end{align*}
Similarly to the classical case, one shows $\mathcal F_W(\partial^\alpha D)(w) = i^{|\alpha|}w^\alpha \mathcal F_W(D)(w)$ for an operator $D$ and a multi-index $\alpha \in \mathbb N_0^{2d}$. Using this, we obtain for $\alpha \in \mathbb N_0^{2d}$, $|\alpha| \leq 2d+1$: 
\begin{align*}
    |\mathcal F_W(\partial^\alpha (A \alpha_z(B^\ast)))(w)| &= |w^\alpha| |\mathcal F_W(A \alpha_z(B^\ast))(w)|
\end{align*}
Evaluating this at $z = z' - \frac{w}{2}$ yields:
\begin{align*}
    |\mathcal F_W(\partial^\alpha(A \alpha_{z -\frac{w}{2}}(B^\ast)))(w)| = |w^\alpha| |V_B(A)(z, w)|
\end{align*}
We can now conclude analogously to the proof of the analogous result for functions, cf.\ the proof of \cite[Thm.\ 14.5.3]{Grochenig2001}:
\begin{align*}
    |V_B A(z, w)| = \frac{1}{|w^\alpha|} |\mathcal F_W(\partial^\alpha (A\alpha_{z-\frac{w}{2}}(B^\ast))(w)|
\end{align*}
Let us estimate the Fourier transform. By the Riemann-Lebesgue Lemma for $\mathcal F_W$, we have:
\begin{align*}
    \| \mathcal F_W(\partial^\alpha (A \alpha_{z-\frac{w}{2}}(B^\ast)))\|_\infty \leq \| \partial^\alpha (A \alpha_{z-\frac{w}{2}}(B^\ast))\|_{\mathcal T^1}
\end{align*}
We can now expand $\partial^\alpha (A \alpha_{z - \frac{w}{2}}(B^\ast))$ into a finite linear combination of expressions of the form $(\partial^{\alpha'} A) (\partial^{\alpha''}\alpha_{z-\frac{w}{2}}(B^\ast))$ with $|\alpha'| + |\alpha''| \leq 2d+1$. Since $\partial^{\alpha'}(A) \in \mathcal L(\mathcal H)$ by assumption and $\partial^{\alpha''} \alpha_{z-\frac{w}{2}}(B^\ast) = \alpha_{z-\frac{w}{2}} \partial^{\alpha''}(B^\ast) \in \mathcal S(\mathcal H)$ (the exchange of the shift and differential uses the property $\alpha_v(\alpha_u(D)) = \alpha_u(\alpha_v(D))$ of the shifts, which is readily verified) and $\| \alpha_{z-\frac{w}{2}}(\partial^{\alpha''}(B^\ast))\|_{\mathcal T^1} = \| \partial^{\alpha''}(B^\ast)\|_{\mathcal T^1}$, we see that $\| \partial^\alpha (A \alpha_{z-\frac{w}{2}}(B^\ast))\|_{\mathcal T^1}$ is uniformly bounded in  $z$ and $w$. Hence, we have seen:
\begin{align*}
    \| A\|_{M^{\infty, 1}, B} \lesssim \int_{\mathbb R^{2d}} \min_{|\alpha| \leq 2d+1} \frac{1}{|w^\alpha|}~dw.
\end{align*}
Finiteness of this integral has been shown in the proof of \cite[Thm.\ 14.5.3]{Grochenig2001}, to which we simply refer.
\end{proof}
Summarizing what we have done so far, we obtained the following inclusions and mapping properties of the Weyl quantization (where the arrow $\overset{\sim}{\longrightarrow}$  indicates a topological isomorphism):
\begin{align*}
    \begin{matrix}
L^\infty(\mathbb R^{2d}) & & \mathcal L(\mathcal H)\\
\hookuparrow & & \hookuparrow\\
M^{\infty, 1}(\mathbb R^{2d}) & \overset{\sim}{\underset{\mathrm{op}^w}{\longrightarrow}} & M^{\infty, 1}(\mathcal H)\\
\hookuparrow & & \hookuparrow\\
C_b^{2d+1}(\mathbb R^{2d}) & & C^{2d+1}(\mathcal H)
    \end{matrix}
\end{align*}
It is not hard to verify that
\begin{align*}
    \{ f \in \mathcal S'(\mathbb R^{2n}):\ &(z \mapsto \alpha_z(f)) \in C^\infty(\mathbb R^{2n}; L^\infty(\mathbb R^{2n}))\}\\
    &= \{ f \in \mathcal S'(\mathbb R^{2n}):\ (z \mapsto \alpha_z(f)) \in C^\infty(\mathbb R^{2n}; C_b^{2d+1}(\mathbb R^{2d}))\}\\
    &= C_b^\infty(\mathbb R^{2n}).
\end{align*}
Using the inclusions on the function side, we therefore obtain:
\begin{align*}
    \{ f \in \mathcal S'(\mathbb R^{2n}):\ (z \mapsto \alpha_z(f)) \in C^\infty(\mathbb R^{2n}; M^{\infty, 1}(\mathbb R^{2n}))\} = C_b^\infty(\mathbb R^{2n}).
\end{align*}
Similarly, one can easily see that
\begin{align*}
    \{ A \in \mathcal S'(\mathcal H):\ &(z \mapsto \alpha_z(A))\in C^\infty (\mathbb R^{2n}; \mathcal L(\mathcal H))\}\\
    &= \{ A \in \mathcal S'(\mathcal H):\ (z \mapsto \alpha_z(A)) \in C^\infty(\mathbb R^{2n}; C^{2d+1}(\mathcal H))\}\\
    &= C^\infty(\mathcal H).
\end{align*}
From the inclusions on the operator side, we therefore obtain
\begin{align*}
        \{ A \in \mathcal S'(\mathcal H):\ (z \mapsto \alpha_z(A)) \in C^\infty(\mathbb R^{2n}; M^{\infty, 1}(\mathcal H))\} = C^\infty(\mathcal H).
\end{align*}
Using that $\mathrm{op}^w$ is a topological isomorphism from $M^{\infty, 1}(\mathbb R^{2d})$ to $M^{\infty, 1}(\mathcal H)$, which also satisfies $\alpha_z(\mathrm{op}^w(f)) =  \mathrm{op}^w(\alpha_{z}(f))$, we see that $\partial^\alpha \mathrm{op}^w(f) = (-1)^{|\alpha|}\mathrm{op}^w(\partial^\alpha f)$ (the factor $(-1)^{\alpha}$ comes from the fact that the partial derivative of $f$ with respect to $\alpha$ is the limit of $(f(\cdot -te_j) - f)/t$, which differs from the standard convention of derivatives by a factor of $-1$). This shows that a phase space-smooth element of $M^{\infty, 1}(\mathbb R^{2d})$, i.e.\ a symbol in $C_b^\infty(\mathbb R^{2d})$, yields a phase space smooth Weyl pseudodifferential operator. The inverse map of $\mathrm{op}^w$ has the same good properties, hence we obtain that a phase space smooth operator is given by a Weyl symbol in $C_b^\infty(\mathbb R^{2d})$.
 
 This finishes the proof of Cordes' theorem for Weyl pseudodifferential operators:
\begin{theorem}\label{thm:CordesWeyl}
    $C^\infty(\mathcal H) = \mathrm{op}^w(C_b^\infty(\mathbb R^{2n}))$.
\end{theorem}

\section{Passing to $\tau$-quantization}\label{sec:2}

We want to note that Cordes proved the result for the Kohn-Nirenberg quantization instead of the Weyl quantization. In terms of quantum harmonic analysis, it is more convenient to work with Weyl pseudodifferential operators. Nevertheless, it is not hard to translate the result we obtained into a result for any other $\tau$-quantization ($\tau \in \mathbb R$). Before describing how this works, we want to recall that $\tau$-quantization is defined by
\begin{align*}
    \mathrm{op}^\tau(f)(\varphi)(t) &= \int_{\mathbb R^d} \int_{\mathbb R^d} f(x\tau + (1-\tau)t, \xi) e^{i\xi\cdot(t-x)} \varphi(x)~dx~d\xi.
\end{align*}
for $f \in \mathcal S(\mathbb R^{2d})$ and $\varphi \in \mathcal S(\mathbb R^d)$. Clearly, $\tau = 0$ yields Kohn-Nirenberg (or left) quantization and $\tau = \frac{1}{2}$ gives Weyl quantization. From a pseudodifferential point of view, the most natural consideration to translate Theorem \ref{thm:CordesWeyl} into a result for $\tau$-quantization is considering the map sending Kohn-Nirenberg symbols to $\tau$-symbols. This map $N_\tau$ is given by
\begin{align*}
    N_\tau f(x, \xi) = \int_{\mathbb R^d} \int_{\mathbb R^d} e^{-iy\cdot \eta} f(x+\tau y, \xi + \eta) ~dy~d\eta 
\end{align*}
for $f \in \mathcal S(\mathbb R^{2d})$ and then extended by continuity to $\mathcal S'(\mathbb R^{2d})$. Then, it holds
\begin{align*}
    \mathrm{op}^0(N_{-\frac{1}{2}}(f)) = \mathrm{op}^w(f), \quad \mathrm{op}^\tau(N_\tau(f)) = \mathrm{op}^0(f).
\end{align*}
Since $N_\tau$ maps $C_b^\infty(\mathbb R^{2d})$ isomorphically to itself, we obtain:
\begin{corollary}\label{cor:tau_quant}
    Let $\tau \in \mathbb R$. Then, it is $C^\infty(\mathcal H) = \mathrm{op}^\tau(C_b^\infty(\mathbb R^{2d}))$.
\end{corollary}
While it is certainly well-known that $N_\tau$ is a topological isomorphism of $C_b^\infty(\mathbb R^{2d})$, this precise statement is somewhat hard to find in the literature and somewhat annoying to prove by the standard methods. We will obtain a proof of this fact by other means, i.e., by adapting of Cordes' theorem for Weyl quantization, which was presented in the previous section, to general $\tau$-quantization. As a by-product, we will also obtain a proof of Corollary \ref{cor:tau_quant}.

This second approach is more in the spirit of quantum harmonic analysis and goes as follows: Instead of the Weyl operators used before, we consider now (again, for $\tau \in \mathbb R$ fixed):
\begin{align*}
    W_{(x, \xi)}^\tau f(t) = e^{i\xi \cdot t - i \tau x\cdot \xi}f(t-x).
\end{align*}
The Weyl operators (cf.\ Eq.~\eqref{eq:weyl_ops}) are recovered by letting $\tau=\frac12$. We will still keep writing $W_z$ for the operators with $\tau = \frac 12$.
Then, the family $(W_{(x, \xi)}^\tau)_{x, \xi \in \mathbb R^{d}}$ is again a strongly continuous projective unitary representation of $\mathbb R^{2d}$ and falls completely within the general framework of QHA on locally compact abelian groups \cite{Fulsche_Galke2023}, see also \cite{Luef_Bastianoni2023}. The multiplier  $m_\tau$ related to the representation, i.e., the function satisfying
\begin{align*}
    W_{(x, \xi)}^\tau W_{(y, \eta)}^\tau = m_\tau((x, \xi), (y, \eta)) W_{(x+y, \xi + \eta)}^\tau
\end{align*}
is given by $m_\tau((x, \xi), (y, \eta)) = e^{-i((1-\tau)x \cdot \eta - \tau y\cdot \xi)}$. Note that the quotient
\begin{align*} 
\frac{m_\tau((x, \xi), (y, \eta))}{m_\tau((y, \eta), (x, \xi))} = e^{i\sigma((x, \xi), (y, \eta))}
\end{align*}
is independent of $\tau$. Further, one readily verifies that
\begin{align*}
    (W_{(x, \xi)}^\tau)^\ast = W_{(-x, -\xi)}^{1-\tau}.
\end{align*}
Now, if we define the shift of an operator $A \in \mathcal L(\mathcal H)$ by $\alpha_{(x, \xi)}^\tau(A) = W_{(x, \xi)}^\tau A (W_{(w, \xi)}^\tau)^\ast$, one immediately sees that this is independent of $\tau$:
\begin{align*}
    W_{(x, \xi)}^\tau A (W_{(w, \xi)}^\tau)^\ast &= W_{(x, \xi)}^\tau A W_{(-x, -\xi)}^{1-\tau} \\
    &= e^{i(\frac 12 - \tau)x\cdot \xi} W_{(x, \xi)}^{1/2} A W_{(-x, -\xi)}^{1/2} e^{i(\frac 12 - 1 + \tau)(-x)\cdot (-\xi)} \\
    &= \alpha_{(x,\xi)}(A).
\end{align*}
Hence, we will simply keep writing $\alpha_{(x, \xi)}(A)$ instead of $\alpha_{(x, \xi)}^\tau(A)$. Next, we consider the Fourier-Weyl transform $\mathcal F_{W}^\tau$ and its inverse $(\mathcal F_W^{\tau})^{-1}$, given by
\begin{align*}
    \mathcal F_W^\tau(A)(x, \xi) &:= \tr(A(W_\xi^\tau)^\ast) = e^{i(\tau - \frac 12)x \cdot \xi} \mathcal F_W(A)(x, \xi),\\
    (\mathcal F_W^{\tau})^{-1}(f) &:= \int_{\mathbb R^{d}} \int_{\mathbb R^d} f(x, \xi) W_{(x, \xi)}^{\tau}~dx ~d\xi = \int_{\mathbb R^d} \int_{\mathbb R^d} e^{i(\frac 12 - \tau) x \cdot \xi}f(x, \xi) W_{(x, \xi)} ~dx~d\xi.
\end{align*}
These formula\lvl{e} can be summarized as $\mathcal F_W^\tau = M^\tau \mathcal F_W^{}$ and $(\mathcal F_W^{\tau})^{-1} = (\mathcal F_W)^{-1} (M^\tau)^\ast$, where $M^\tau$ is the multiplication operator $f(x, \xi) \mapsto e^{i(\tau - \frac{1}{2})x \cdot \xi} f(x, \xi)$. Since $M^\tau$ is a topological, respectively isometrical isomorphism of $\mathcal S(\mathbb R^{2d})$, $\mathcal S'(\mathbb R^{2d})$ and $L^2(\mathbb R^{2d})$, $\mathcal F_W^\tau$ has the same good mapping properties as $\mathcal F_W$, i.e.\ it maps $\mathcal S(\mathcal H)$ to $\mathcal S(\mathbb R^{2d})$, $\mathcal T^2(\mathcal H)$ to $L^2(\mathbb R^{2d})$ and $\mathcal S'(\mathcal H)$ to $\mathcal S'(\mathbb R^{2d})$. Exactly as in Lemma \ref{lem:WeylQuant}, one proves that:
\begin{align*}
    (\mathcal F_W^\tau)^{-1} \circ \mathcal F_\sigma = \mathrm{op}^\tau.
\end{align*}
It is not hard to see that $\mathcal F_W^\tau(\alpha_z(A)) = \gamma_z (\mathcal F_W^\tau(A))$ holds true for every value of $\tau$. When it comes to the modulation of operators, the situation unfortunately becomes less nice for $\tau \neq \frac{1}{2}$. We consider $\gamma_z^\tau(A) = W_{\frac{z}{2}}^{\tau} A W_{\frac{z}{2}}^\tau$. We now have, using the identity $W_{(x, \xi)}^\tau = e^{i(\frac 12 - \tau) x\cdot \xi} W_{(x ,\xi)}$:
\begin{align*}
    \mathcal F_W^\tau(\gamma_{(x, \xi)}^\tau(A))(y, \eta) &= e^{i(\frac{1}{2}-\tau )(x\cdot \xi/2 - y\cdot \eta)}\mathcal F_W(\gamma_{(x, \xi)}(A))(y, \eta)\\
    &=e^{i(\frac{1}{2} - \tau)(x\cdot \xi/2 - y\cdot \eta)} \alpha_{(x, \xi)}(\mathcal F_W(A))(y, \eta)\\
    &= e^{\frac{3i}{2}(\frac 12 -\tau) x \cdot \xi}e^{i\sigma(((\frac 12 -\tau)x, (\tau - \frac 12)\xi), (y, \eta))} \alpha_{(x, \xi)}(\mathcal F_W^{\tau}(A))(y, \eta)\\
    &= e^{\frac{3i}{2}(\frac 12 -\tau) x \cdot \xi} \gamma_{((\frac{1}{2}-\tau)x, (\tau - \frac{1}{2})\xi)} \alpha_{(x, \xi)} \mathcal F_W^\tau(A)(y, \eta)
\end{align*}
Now, with the $\tau$-dependent version of the Plancherel theorem (which is readily established by using $\mathcal F_W^\tau = M^\tau \mathcal F_W$), we obtain for $A \in \mathcal S'(\mathcal H)$ and $B \in \mathcal S(\mathcal H)$ (where $w = (y, \eta)$ and $z = (x, \xi)$):
\begin{align*}
    |V_B A(z, w)| &= |\langle A, \gamma_w \alpha_z(B)\rangle|\\
    &= |\langle \mathcal F_W^\tau(A),  \gamma_{((\frac{1}{2}-\tau)y, (\tau - \frac{1}{2})\eta)}\alpha_{(y, \eta)} \gamma_{(x, \xi)} \mathcal F_W^\tau(B)\rangle|\\
    &= |\langle \mathcal F_\sigma \mathcal F_W^\tau(A), \alpha_{((\frac{1}{2}-\tau)y, (\tau - \frac{1}{2})\eta)}\gamma_{(y, \eta)} \alpha_{(x, \xi)} \mathcal F_\sigma \mathcal F_W^\tau(B)\rangle|\\
    &= | \langle \mathcal F_\sigma \mathcal F_W^\tau (A), \gamma_{(y, \eta)} \alpha_{(x + (\tau - \frac{1}{2})y, \xi - (\tau - \frac{1}{2})\eta)} \mathcal F_\sigma \mathcal F_W^\tau(B)\rangle|.
\end{align*}
These identities show that $A = \mathrm{op}^\tau(f) \in M^{\infty, 1}(\mathcal H)$ if and only if $f \in M^{\infty, 1}(\mathbb R^{2d})$, and $f \mapsto \mathrm{op}^\tau(f)$ is (upon choosing an appropriate norm) an isometric isomorphism from $M^{\infty, 1}(\mathbb R^{2d})$ to $M^{\infty, 1}(\mathcal H)$. In particular, we now have the isometric isomorphisms
\begin{align*}
    M^{\infty, 1}(\mathcal H) \overset{(\mathrm{op}^\tau)^{-1}}{\longrightarrow} M^{\infty, 1}(\Xi) \overset{\mathrm{op}^w}{\longrightarrow} M^{\infty, 1}(\mathcal H).
\end{align*}
Therefore, the same reasoning as presented at the end of Section \ref{sec:2} yields topological isomorphisms:
\begin{align*}
    C^\infty(\mathcal H) \overset{(\mathrm{op}^{\tau})^{-1}}{\longrightarrow} C_b^\infty(\Xi) \overset{\mathrm{op}^w}{\longrightarrow} C^\infty(\mathcal H)
\end{align*}
This proves Corollary \ref{cor:tau_quant}. As an immediate consequence, we get:
\begin{corollary}
$N_\tau$ is a topological isomorphism of $C_b^\infty(\mathbb R^{2d})$.
\end{corollary}

\section{Heisenberg-smooth operators on abelian Lie groups}\label{sec:3}

The above reasoning for general $\tau$ is not restricted to the phase space $\mathbb R^{2d}$, but can be naturally extended to phase spaces $\Xi = G \times \widehat{G}$, where $G$ is a commutative Lie group. As is well-known, such groups are of the form $G \cong \mathbb R^{d_1} \times \mathbb T^{d_2} \times D$, where $d_1, d_2 \in \mathbb N_0$ and $D$ is a discrete abelian group. Since $D$ only contributes to the smooth structure of the phase space when its Pontryagin dual $\widehat{D}$ is a Lie group of dimension $> 0$, there is no loss of generality in assuming that $D = \mathbb Z^{d_3}$ with $d_3 \in \mathbb N_0$. The dual group $\widehat{G}$ is then $\widehat{G} \cong \mathbb R^{d_1} \times \mathbb Z^{d_2} \times \mathbb T^{d_3}$. During this section, we will keep writing the group additively, while at least on $\mathbb T$, we, of course, have a multiplication. Indeed, the tools and methods of quantum harmonic analysis work on general locally compact abelian groups -- see, e.g., \cite{Fulsche_Galke2023}. We will only briefly sketch the ideas involved to show how the method presented on $\Xi = \mathbb R^{2d}$ carries essentially over to the case at hand. 

Of course, we will always assume that $d_1 + d_2 + d_3 > 0$. The phase space will, depending on what is needed, either be written as $\Xi = G \times \widehat{G}$ or $\Xi = \mathbb R^{2d_1} \times \mathbb T^{d_2 + d_3} \times \mathbb Z^{d_2 + d_3}$. 

Indeed, one can even be more general than $\tau$-quantization. More precisely, we fix any continuous group homomorphism $\Phi$ of the position space $G$. Then, we can consider the operators $W_{(x, \xi)}^\Phi$, $(x, \xi) \in \Xi = G \times \widehat{G}$, acting on $\mathcal H = L^2(G)$ by
\begin{align*}
    W_{(x, \xi)}^\Phi f(t) = \overline{\langle \xi, \Phi(x) \rangle} \langle  \xi, t\rangle f(t-x),\quad t \in G.
\end{align*}
Now, the phase space $\Xi$ acts on $\mathcal L(\mathcal H)$ by $\alpha_z(A) = W_z^\Phi A (W_z^\Phi)^\ast$, which is again independent of the particular choice of $\Phi$. Further, we define the Fourier-Weyl transform $\mathcal F_W^\Phi(A)$ by $\mathcal F_W^\Phi(A)(w) = \tr(A(W_w^\Phi)^\ast)$ whenever $A \in \mathcal T^1(\mathcal H)$.

The Schwartz space on $G$, $\mathcal S(G)$, is now given as the tensor product $\mathcal S(G) = \mathcal S(\mathbb R^{d_1}) \otimes \mathcal S(\mathbb T^{d_2}) \otimes \mathcal S(\mathbb Z^{d_3})$, where $\mathcal S(\mathbb R^{d_1})$ is defined as usual, $\mathcal S(\mathbb T^{d_2}) = C^\infty(\mathbb T^{d_2})$ and $\mathcal S(\mathbb Z^{d_3}) = \mathfrak s(\mathbb Z^{d_3})$, the rapidly decreasing sequences (i.e., sequences decreasing faster than every polynomial). Similarly, $\mathcal S(\Xi)$ is defined as a tensor product. The dual spaces of tempered distributions are then $\mathcal S'(G)$ and $\mathcal S'(\Xi)$, respectively. An important fact that we will implicitly make use of is the following:
\begin{lemma}\label{lemma:phi1}
    Let $\Phi$ be a continuous group homomorphism of $G$. Then, the function
    \begin{align*}
        H_\Phi(x, \xi) = \overline{\langle \xi, \Phi(x)\rangle}
    \end{align*}
    is a multiplier of $\mathcal S(\Xi)$, $L^2(\Xi)$ and $\mathcal S'(\Xi)$.
\end{lemma}
\begin{proof}
    We only give a sketch of the proof. Writing the group $G$ as $G = G_1 \times G_2 \times G_3$, where $G_1 = \mathbb R^{d_1}$, $G_2 = \mathbb T^{d_2}$, $G_3 = \mathbb Z^{d_3}$, we can understand $\Phi$ as a ``matrix'' $\Phi = (\Phi_{j,k})_{j,k = 1, 2, 3}$, where $\Phi_{j,k} \in \operatorname{Hom}(G_k, G_j)$. Now, the form of these homomorphisms $\Phi_{j,k}$ can be described explicitly:
    \begin{itemize}
        \item $\Phi_{1,1}$ is given by multiplication with a $d_1 \times d_1$ matrix $A$: $\Phi_{1,1}x = Ax$.
        \item Since $\mathbb T^{d_2}$ is compact, the range of $\Phi_{1,2}$ has to be a compact subgroup of $\mathbb R^{d_1}$. Hence, $\Phi_{1,2} = 0$.
        \item $\Phi_{1,3}$ is again given by multiplication with a $d_1 \times d_3$ matrix $B$, $\Phi_{1,3} k = Bk$.
        \item Homomorphisms from $\mathbb R^{d_1}$ to $\mathbb T^{d_2}$ are described by a $d_2 \times d_1$ matrix $\Theta$ and the map is given by $x \mapsto e^{i\Theta x}$, where $e^{i\Theta x}$ has to be understood as the tuple where the $j$-th entry has to be understood as the exponential of $j-th$ entry of $i\Theta x$.
        \item Homomorphisms of $\mathbb T^{d_2}$ are given by multiplication with a $d_2\times d_2$ matrix $C$ with integer entries.
        \item Every homomorphism from $\mathbb Z^{d_3}$ to $\mathbb T^{d_2}$ can be uniquely described by a $d_2\times d_3$ matrix  $\zeta = (\zeta_{j,m})_{j=1, \dots, d_2, m = 1, \dots d_3}$ with elements in $\mathbb T$. The action is 
        \begin{align*}
            \begin{pmatrix}
                k_1\\
                \vdots\\
                k_{d_3}
            \end{pmatrix} \mapsto \begin{pmatrix}
                \zeta_{1,1}^{k_1} \cdot \dots \cdot \zeta_{1, d_3}^{k_{d_3}}\\
                \vdots\\
                \zeta_{d_2, 1}^{k_1} \cdot \dots \cdot \zeta_{d_2, d_3}^{k_{d_3}}
            \end{pmatrix}.
        \end{align*}
       We simply write this as $\zeta k$. 
        \item Since $\mathbb R^{d_1}$ and $\mathbb T^{d_2}$ are connected, their images under $\Phi_{3,1}$ and $\Phi_{3,2}$ are connected subgroups of $\mathbb Z^{d_3}$. Hence, $\Phi_{3,1} = 0$ and $\Phi_{3,2} = 0$.
        \item $\Phi_{3,3}$ is given by multiplication by a matrix $D$ with integer entries.
    \end{itemize}
    In total, 
    \begin{align*}
        \Phi = \begin{pmatrix}
            A & 0 & B\\
            e^{i\Theta (\cdot)} & C & \zeta {(\cdot)}\\
             0 & 0 & D
        \end{pmatrix}: \mathbb R^{d_1} \times \mathbb T^{d_2} \times \mathbb Z^{d_3} \to \mathbb R^{d_1} \times \mathbb T^{d_2} \times \mathbb Z^{d_3}.
    \end{align*}
    Having this explicit form of $\Phi$ at hand, it is not hard to verify that $H_\Phi$ is smooth and has derivatives that are polynomially bounded. Hence, it is a multiplier of $\mathcal S(\Xi)$ and, therefore, also of $\mathcal S'(\Xi)$. It is clear that this is also a multiplier of $L^2(\Xi)$, as $| H_\Phi| = 1$. 
\end{proof}
We now note the two facts, the simple verification of which is left to the reader.
\begin{lemma}\label{lemma:Phi2}
    Let $\Phi: G \to G$ be a continuous homomorphism. 
    \begin{enumerate}[(1)] 
    \item For every $(x, \xi) \in \Xi$ we have $(W_{(x, \xi)}^\Phi)^\ast = W_{(-x, -\xi)}^{I - \Phi}$.
    \item $\mathcal F_W^\Phi(A)(x, \xi) = \langle \xi, \Phi(x)\rangle \mathcal F_W^0(A)(x, \xi)$.
    \end{enumerate}
\end{lemma}
We denote by $\mathcal S(\mathcal H)$ the space of all linear operators mapping $\mathcal S'(G) \to \mathcal S(G)$ continuously. Since $\mathcal S(G)$ is a nuclear space, the embedding $\iota: \mathcal S(G) \hookrightarrow L^2(G)$ is nuclear \cite[Theorem 50.1(c)]{Treves1967}. Therefore, every $A \in \mathcal S(\mathcal H)$ can be identified with the operator $\iota \circ A|_{L^2(G)} \in \mathcal L(\mathcal H)$, which is then trace class. Using that every such $A$ is given by an integral kernel in $\mathcal S(G \times G)$, direct computations with the integral kernels of $A$ and the integral kernels of the Weyl operators show that $A \in \mathcal S(\mathcal H)$ if and only if $\mathcal F_W^{0}(A) \in \mathcal S(\Xi)$. By Lemmas \ref{lemma:phi1} and \ref{lemma:Phi2}, we therefore obtain that $\mathcal F_W^\Phi: \mathcal S(\mathcal H) \to \mathcal S(\Xi)$ is a topological isomorphism (since multiplication by $H_\Phi$ is a multiplier of $\mathcal S(\Xi)$ and clearly also invertible). The same goes for the inverse map $(\mathcal F_W^{\Phi})^{-1}: f \mapsto \int_\Xi f(z) W_z^\Phi dz$. Here, $dz$ is the product measure of some Haar measure on $G$ and its dual Plancherel measure on $\widehat{G}$. By duality, the Fourier transforms extend to topological isomorphisms $\mathcal F_W^\Phi: \mathcal S'(\mathcal H) \to \mathcal S'(\Xi)$ and $(\mathcal F_W^\Phi)^{-1}: \mathcal S'(\Xi) \to \mathcal S'(\mathcal H)$. Here, $\mathcal S'(\mathcal H)$ is the set of all continuous linear operators mapping $\mathcal S(G) \to \mathcal S'(G)$.

The operators $W_z^\Phi$ still form an irreducible projective unitary representation of $\Xi$, i.e.\ they satisfy $W_z^\Phi W_w^\Phi = m_\Phi(z, w) W_{z+w}^\Phi$. Then, $\sigma(z, w) = m_\Phi(z, w)/m_\Phi(w, z)$ serves as our substitute for the symplectic form. We define the symplectic Fourier transform by
\begin{align*}
    \mathcal F_\sigma(f)(w) = \int_\Xi \sigma(z,w)f(z)~dz.
\end{align*}
 We define $\Phi$-quantization as the map from $\mathcal S'(\Xi)$ to $\mathcal S'(\mathcal H)$ which is formally given by
 \begin{align*}
     \mathrm{op}^\Phi(f)(\varphi)(x) = \int_{\widehat{G}}\int_G f(x - \Phi(x-y), \xi) \varphi(y) \langle x-y, \xi\rangle ~dy~d\xi.
 \end{align*}
 Just as in Lemma \ref{lem:WeylQuant}, one proves that:
 \begin{lemma}
     $\operatorname{op}^\Phi = (\mathcal F_W^{\Phi})^{-1} \circ \mathcal F_\sigma$.
 \end{lemma}
 As special instances, we have the $\tau$-quantization for each $\tau \in \mathbb Z$ on $G$ or the $A$-quantization on $G = \mathbb R^d$, where $A \in \mathbb R^{d \times d}$ and $\mathrm{op}^A(f)$ acts by
 \begin{align*}
     \mathrm{op}^A(f) &= \int_{\mathbb R^d}\int_{\mathbb R^d} f(x - A(x-y), \xi)\varphi(y) e^{i\xi(x-y)}~dy~d\xi.
 \end{align*}
See, e.g., \cite{Toft2017} for a general discussion concerning this $A$-quantization.

The covariance property $\alpha_z(\mathrm{op}^\Phi(f)) = \mathrm{op}^\Phi(\alpha_z(f))$ is still satisfied, as an immediate consequence, where of course $\alpha_z(f)(w) = f(w-z)$. Identifying $\Xi$ with $\widehat{\Xi}$ through $z \mapsto \sigma(z, \cdot)$, we can write the action of the dual phase space as $\gamma_z(f)(w) = \sigma(z, w)f(w)$. We can now introduce the STFT of $f \in \mathcal S(\Xi)$ with respect to the window $0 \neq g \in \mathcal S(\Xi)$ by
\begin{align*}
    V_g f(z, w) = \langle f, \gamma_w \alpha_z(g)\rangle
\end{align*}
and define $M^{\infty, 1}(\Xi)$ as the space of all $f \in \mathcal S'(\Xi)$ such that $V_g f \in L^{\infty, 1}(\Xi)$, i.e.,
\begin{align*}
    \int_\Xi \sup_{z \in \Xi} |V_g f(z, w)| ~dw < \infty.
\end{align*}
When defining the modulation of an operator, we now have to accept that dividing by $2$ is, in general, not possible. So instead of multiplying the operator with $W_{\frac{z}{2}}$ from each side, we set $\gamma_z^\Phi(A) = W_{z}^\Phi A$. One defines $M^{\infty, 1}(\mathcal H)$ to be the set of all $A \in \mathcal S'(\mathcal H)$ such that $\langle A, \gamma_w \alpha_z(B)\rangle \in L^{\infty, 1}(\Xi)$. By computations similar to the case $\Xi = \mathbb R^{2d}$, one proves that $A = \mathrm{op}^\Phi(f) \in M^{\infty, 1}(\mathcal H)$ if and only if $f \in M^{\infty, 1}(\Xi)$ and $f \mapsto \mathrm{op}^\Phi(f)$ is an isometric isomorphism (upon choosing appropriate norms). Note that, as an immediate consequence, we obtain that $M^{\infty, 1}(\Xi)$ is invariant under the map $\mathcal F_\sigma \circ \mathcal F_W^{\Phi'} \circ (\mathcal F_W^{\Phi})^{-1} \circ \mathcal F_\sigma$ for two different choices $\Phi, \Phi' \in \operatorname{Hom}(G)$, which maps a $\Phi$-symbol to the corresponding $\Phi'$-symbol.

By this, we can reduce the question about classifying the Heisenberg-smooth operators on $L^2(G)$ to the case of $\Phi$ being the zero morphism, i.e., once we prove that the space of Heisenberg-smooth operators on $L^2(G)$ agrees with $\mathrm{op}^{0}(C_b^\infty(\Xi))$, we know just as in Section \ref{sec:3} that this class equals $\mathrm{op}^\Phi(C_b^\infty(\Xi))$ for each $\Phi \in \operatorname{Hom}(G)$. 

The last thing one has to discuss is the appropriate embedding results. 

\begin{lemma}
    $M^{\infty,1}(\mathcal H) \hookrightarrow \mathcal L(\mathcal H)$.
\end{lemma}
\begin{proof} This was proven in \cite[Theorem 5.2]{Grochenig_Strohmer2007} by showing that $\| \mathrm{op}^0(f)\|_{op} \leq \| f\|_{M^{\infty, 1}(\Xi)}$. 
\end{proof}

For functions $f: \Xi \to \mathbb C$, where $\Xi \cong \mathbb R^{2d_1} \times \mathbb T^{d_2 + d_3} \times \mathbb Z^{d_2 + d_3}$, we can consider derivatives in the directions of $\mathbb R^{2d_1}$-variables and in the direction of $\mathbb T^{d_2 + d_3}$-variables. Denote by $\partial_j f$ the partial derivative in direction of $e_j \in \mathbb R^{2d_1}$, $j = 1, \dots, 2d_1$, and by $\partial_j' f$ the partial derivative in direction of the $j$-th variable of $\mathbb T^{d_2 + d_3}, j = 1, \dots, d_2 + d_3$. For higher order derivatives, we let for $\alpha \in \mathbb N_0^{2d}$ and $\beta \in \mathbb N_0^{d_2 + d_3}$ denote by $(\partial ')^\beta \partial^\alpha f$ the corresponding iterated derivative. Now, we clearly have (with the standard multi-index notation), 
\begin{align*}
    \mathcal F_\sigma((\partial')^\beta \partial^\alpha f)(x, \theta, k) \cong C_{\alpha, \beta} x^\alpha k^\beta \mathcal F_\sigma(f)(x, \theta, k), \quad (x, \theta, k) \in \mathbb R^{2d} \times \mathbb T^{d_2 + d_3} \times \mathbb Z^{d_2 + d_3}.
\end{align*}
Here, $C_{\alpha, \beta}$ is a unimodular constant. We consider the class of symbols
\begin{align*}
    C_b^{2d_1+1, 2d_2 + 2d_3}(\Xi) = \{ f: \Xi \to \mathbb C; ~&f \text{ is } 2d_1+1\text{-times differentiable in } \mathbb R^{2d_1} \\
    & \text{and } 2d_2+2d_3\text{-times differentiable in } \mathbb T^{d_2 + d_3} \\
    & \text{with bounded continuous derivatives}\,\}.
\end{align*}
Clearly, $C_b^{2d_1 + 1, 2d_2 + 2d_3}(\Xi)$ is a Banach space if it is equipped with the obvious choice of norm. 

\begin{lemma}
    $C_b^{2d_1 + 1, 2d_2 + 2d_3}(\Xi) \hookrightarrow M^{\infty, 1}(\Xi)$.
\end{lemma}
\begin{proof} As in the case of $\Xi = \mathbb R^{2d}$ (see again \cite[Theorem 14.5.3]{Grochenig2001}), the STFT $V_gf$ of $f \in C_b^{2d_1 + 1, 2d_2 + 2d_3}(\Xi)$ with window $0 \neq g \in \mathcal S(\Xi)$ can be estimated (with $w = (x, \theta, k)$) by
\begin{align*}
    |V_g f(z, w)| &\cong \frac{1}{|x^\alpha k^\beta|} |\mathcal F_\sigma ((\partial')^\beta \partial^\alpha (f \cdot \alpha_z(\overline{g}))(w)|\\
    &\lesssim \frac{1}{|x^\alpha k^\beta|} \| f\|_{C_b^{2d_1 + 1, 2d_2 + 2d_3}}.
\end{align*}
Hence, for such $f$ we have
\begin{align*}
    \| V_g f\|_{L^{\infty, 1}} &\lesssim \| f\|_{C_b^{2d_1 + 1, d_2 + d_3}} \int_{\Xi} \min_{|\alpha| \leq 2d_1+1, |\beta| \leq 2d_2 + 2d_3} \frac{1}{|x^\alpha| |k^\beta|} ~dw\\
    &\leq \| f\|_{C_b^{2d_1 + 1, d_2 + d_3}} \int_{\mathbb R^{2d_1}} \min_{|\alpha| \leq 2d_1 + 1} \frac{1}{|x^\alpha|}~dx \sum_{k \in \mathbb Z^{d_2 + d_3}} \min_{|\beta| \leq 2d_2 + 2d_3} \frac{1}{|k^\beta|}.
\end{align*}
The first integral is, as already noted earlier, known to be finite. For the infinite sum, we have, abbreviating $d' = d_2 + d_3$,
\begin{align*}
    \sum_{k \in \mathbb Z^{d'}} \min_{|\beta| \leq 2d'} \frac{1}{|k^\beta|} &\leq \sum_{k_1, \dots, k_{d'} \neq 0} \frac{1}{k_1^2 \cdot \hdots \cdot k_{d'}^2} + d'\sum_{k_1 = 0, k_2, \dots, k_{d'} \neq 0} \frac{1}{k_1^0 k_2^2 \dots k_{d'}^2}   \\
    &\quad + \binom{d'}{2}\sum_{k_1 = k_2 = 0, k_3, \dots k_{d'} \neq 0} \frac{1}{k_1^0 k_2^0 k_3^2 \dots k_{d'}^2} \\
    &\quad \dots + d'\sum_{k_1 = \dots = k_{d' - 1} = 0, k_{d'} \neq 0} \frac{1}{k_1^0 \dots k_{d' - 1}^0 k_{d'}^2} + \sum_{k_1 = \dots = k_{d'} = 0} \frac{1}{k_1^0 \dots k_{d'}^0}\\
    &\quad < \infty.
\end{align*}
Here, we wrote the $0$-exponents explicitly to clarify how the multi-index $\beta$ was chosen for the particular subset of $\mathbb Z^{d'}$ over which we are summing. This shows $C_b^{2d_1 + 1, 2d_2 + 2d_3}(\Xi) \hookrightarrow M^{\infty, 1}(\Xi)$.
\end{proof}

Analogously, we can now define the spaces $C^{k, m}(\mathcal H)$ of operators which are $k$-times continuously differentiable in $\mathbb R^{2d}$-directions and $m$-times continuously differentiable in $\mathbb T^{d_2 + d_3}$-directions. $C_b^\infty(\Xi)$ and $C^\infty(\mathcal H)$ are now, respectively, the intersections of all the spaces $C_b^{k, m}(\Xi)$ and $C^{k, m}(\mathcal H)$. Of course, results such as the product rule and the quotient rule hold for the differentiation of operators on $\mathcal H$, proving that $C^{k,m}(\mathcal H)$ as well as $C^\infty(\mathcal H)$ are spectrally invariant subalgebras of $\mathcal L(\mathcal H)$.

\begin{lemma}
    The following inclusions hold true continuously: $C^{2d_1 + 1, 2d_2 + 2d_3}(\mathcal H) \hookrightarrow M^{\infty, 1}(\mathcal H)$ and $M^{\infty, 1}(\Xi) \hookrightarrow L^\infty(\Xi)$.
\end{lemma}
\begin{proof} 
The embedding $C^{2d_1 + 1, 2d_2 + 2d_3}(\mathcal H) \hookrightarrow M^{\infty, 1}(\mathcal H)$ is obtained similarly to the embedding $C_b^{2d_1, 2d_2 + 2d_3}(\Xi) \hookrightarrow M^{\infty, 1}(\Xi)$, compare with Lemma \ref{lem:OCM} above. The inclusion $M^{\infty, 1}(\Xi) \hookrightarrow L^\infty(\Xi)$ can be proven just as for $\Xi = \mathbb R^{2d}$, cf.\ \cite[Proposition 1.7]{Toft2004}.
\end{proof}

We can summarize our findings for general phase space $\Xi = G \times \widehat{G}$, $G = \mathbb R^{d_1} \times \mathbb T^{d_2} \times \mathbb Z^{d_3}$, by the following diagram:

\begin{align*}
    \begin{matrix}
L^\infty(\Xi) & & \mathcal L(\mathcal H)\\
\hookuparrow & & \hookuparrow\\
M^{\infty, 1}(\Xi) & \overset{\sim}{\underset{\mathrm{op}^\Phi}{\longrightarrow}} & M^{\infty, 1}(\mathcal H)\\
\hookuparrow & & \hookuparrow\\
C_b^{2d_1+1, 2d_2 + 2d_3}(\Xi) & & C^{2d_1+1, 2d_2 + 2d_2}(\mathcal H)
    \end{matrix}
\end{align*}
Applying now the same reasoning as for $\Xi = \mathbb R^{2d}$, we obtain:
\begin{theorem}
    Let $G = \mathbb R^{d_1} \times \mathbb T^{d_2} \times \mathbb Z^{d_3}$ and set $\Xi = G \times \widehat{G}$. For $\Phi \in \operatorname{Hom}(G)$ we have, with the notation from above: $C^\infty(\mathcal H) = \mathrm{op}^\Phi(C_b^\infty(\Xi))$.
\end{theorem}

We want to emphasize that the above result is already known for $G = \mathbb T$ and $\Phi = 0$, as it was proven in \cite{Melo1997} by different means.

Again, we obtain as a consequence:
\begin{corollary} 
Let $\Phi, \Phi' \in \operatorname{Hom}(G)$. The map sending $\Phi$-symbols to $\Phi'$-symbols, $\mathcal F_\sigma \circ (\mathcal F_W^{\Phi'})^{-1} \circ \mathcal F_W^\Phi \circ \mathcal F_\sigma$, is a topological isomorphism of $C_b^\infty(\Xi)$.
\end{corollary}

We just mention at the end that, indeed, none of the statements we made hinged on the fact that we consider operators on Hilbert spaces. Indeed, the same characterization of Heisenberg-smooth operators holds true for operators acting on modulation spaces $M^p(G)$ instead of $L^2(G)$ (note that $M^2(G) = L^2(G)$). Without going into many details on this, we just mention that the spaces $M^p(G)$ are essentially defined as the space of all functions $f$ on $G$ such that the STFT $V_g(f)$ is $p$-integrable on $\Xi$. Then, every symbol in $M^{\infty, 1}(\Xi)$ yields a Kohn-Nirenberg pseudodifferential operator which even acts boundedly on $M^p(G)$ with $\| \mathrm{op}(f)\|_{M^p(G)\to M^p(G)} \lesssim \| f\|_{M^{\infty, 1}(\Xi)}$ (cf.\ \cite[Theorem 5.2]{Grochenig_Strohmer2007}). Since $M^{\infty, 1}(\Xi)$ is invariant under the change of quantization, as mentioned above, the statement holds for any quantization $\mathrm{op}^\Phi$. Defining the class $M^{\infty, 1}(M^p(G))$ of operators on $M^p(G)$ analogously to $M^{\infty, 1}(\mathcal H)$ (possibly with replacing the Hilbert-Schmidt inner product by the bilinear pairing between nuclear and bounded operators), one therefore has the embedding $M^{\infty, 1}(M^p(G)) \hookrightarrow \mathcal L(M^p(G))$. On the other hand, since the operators $W_z^\Phi$ leave the modulation spaces $M^p(G)$ invariant, we can define the classes $C^{k, m}(M^p(G))$ of differentiable operators analogously. Then, the embedding $C^{2d_1 + 1, 2d_2 + 2d_3}(M^p(G)) \hookrightarrow M^{\infty, 1}(M^p(G))$ can also be proven analogously (replacing the Riemann-Lebesgue Lemma for $\mathcal F_W$ by the estimate $|\tr(AN)| \leq \| A\|_{op} \| N\|_{\mathcal N}$, where $\| \cdot \|_{\mathcal N}$ is the nuclear norm). Therefore, we just mention that we obtain analogously the following result. Here, we stay within the reflexive range $1 < p < \infty$, so we have no additional trouble from the lack of reflexivity in the end-point cases, which would, for example, destroy the duality between the nuclear and the bounded operators.
\begin{theorem}
    Let $G = \mathbb R^{d_1} \times \mathbb T^{d_2} \times \mathbb Z^{d_3}$ and set $\Xi = G \times \widehat{G}$. Let $1 < p < \infty$ and $\Phi \in \operatorname{Hom}(\Xi)$. Then, we have: $C^\infty(M^p(G)) = \mathrm{op}^\Phi(C_b^\infty(\Xi))$.
\end{theorem}

\section{Schatten-class versions of the Calder\'{o}n-Vaillancourt and Cordes theorems}\label{sec:5}

Let us return to the case of $G = \mathbb R^d$ and Weyl quantization. Results analogous to those of this section can again be obtained in the more general setting $G = \mathbb R^{d_1} \times \mathbb T^{d_2} \times \mathbb Z^{d_3}$ and $\Phi \in \operatorname{Hom}(G)$ with the same methods, but we leave this to the interested reader. We will now discuss how to obtain Schatten-class versions of the result on Heisenberg-smooth operators. We have to adapt our methods and step away from modulation spaces, as one of the inclusions presented before seems unavailable. Indeed, while we still have the inclusion $M^{p, 1}(\mathcal H) \hookrightarrow \mathcal T^p(\mathcal H)$ for each $p \in [1, \infty)$ (this follows from $M^{\infty, 1}(\mathcal H) \hookrightarrow \mathcal L(\mathcal H)$ and the well-known inclusion $M^1(\mathcal H) \hookrightarrow \mathcal T^1(\mathcal H)$, combined with complex interpolation), it is not clear to us if one can obtain inclusions of the form $W^{k,p}(\mathcal H) \hookrightarrow M^{p, 1}(\mathcal H)$ (with $W^{k,p}(\mathcal H)$ an appropriate Sobolev space, to be discussed below). Indeed, the proof of $C^{2d+1}(\mathcal H) \hookrightarrow M^{\infty, 1}(\mathcal H)$ hinged on the Riemann-Lebesgue Lemma for the Fourier transform, $\| \mathcal F_W(A)\|_\infty \leq \| A\|_{\mathcal T^1}$, and it is well-known that no estimates of the form $\| \mathcal F_W(A)\|_{L^p} \lesssim \| A\|_{T^q}$ are available for $p < 2$ and $q$ the conjugate exponent. Therefore, we will instead make use of estimates which are, in general, a little rougher than the $M^{\infty, 1}$-embeddings, namely Calder\'{o}n-Vaillancourt estimates.

Let us recall the notion of convolution of operators and functions, as it is used in quantum harmonic analysis, cf.\ \cite{werner84, keyl_kiukas_werner16, luef_eirik2018, Fulsche_Galke2023}. Given an operator $A \in \mathcal T^1(\mathcal H)$ and $f \in L^1(\mathbb R^{2d})$ we set
\begin{align*}
    f \ast A := \int_{\mathbb R^{2d}} f(z) \alpha_z(A) ~dz.
\end{align*}
Then, $f \ast A \in \mathcal T^1(\mathcal H)$ with $\| f\ast A\|_{\mathcal T^1} \leq \| f\|_{L^1} \| A\|_{\mathcal T^1}$. This convolution extends to $f \in L^\infty(\mathbb R^{2d})$ with $f \ast A \in \mathcal L(\mathcal H)$ and $\| f \ast A\|_{op} \leq \| f\|_\infty \| A\|_{\mathcal T^1}$. 
More generally, we obtain by complex interpolation that $\| f \ast A\|_{\mathcal T^p} \leq \| f\|_{L^p} \| A\|_{\mathcal T^1}$. This convolution extends to a convolution between $\mathcal S'(\mathbb R^{2d})$ and $\mathcal S(\mathcal H)$: If $f \in \mathcal S'(\mathbb R^{2d})$ and $A \in \mathcal S(\mathcal H)$, then $f \ast A$ is well-defined as an object in $\mathcal S'(\mathcal H)$. 

Analogously, one can define the convolution between operators. For $A, B \in \mathcal T^1(\mathcal H)$ we define the function $A \ast B(z)$, $z \in \RR^{2d}$, by $A \ast B(z) = \tr(A \alpha_z \beta_-(B))$. Here, $\beta_-(B) = UBU$, where $U$ is the parity operator on $\mathcal H = L^2(\mathbb R^d)$: $U\varphi(x) = \varphi(-x)$. One obtains that $A \ast B \in L^1(\RR^{d})$ with $\| A \ast B\|_{L^1} \leq (2\pi)^d \| A\|_{\mathcal T^1} \| B\|_{\mathcal T^1}$ and $\int_{\RR^{2d}} A \ast B(z) ~dz = (2\pi)^d \tr(A) \tr(B)$. 
Again, when $A \in \mathcal L(\mathcal H)$ and $B \in \mathcal T^1(\mathcal H)$, then we analogously have (with $A \ast B$ defined by the same formula) $A \ast B \in L^\infty(\RR^{2d})$ with $\| A \ast B\|_\infty \leq \| A\|_{op} \| B\|_{\mathcal T^1}$.
By complex interpolation, we again see $\| A \ast B\|_p \leq (2\pi)^{d/p}\| A\|_{\mathcal T^p} \| B\|_{\mathcal T^1}$. These convolutions behave nicely with the Fourier transforms of quantum harmonic analysis. 
For example, $\mathcal F_\sigma(A \ast B) = \mathcal F_W(A) \cdot \mathcal F_W(B)$ and $\mathcal F_W(f \ast A) = \mathcal F_\sigma(f) \cdot \mathcal F_W(A)$. As a consequence, we obtain $\operatorname{op}^w(f) \ast \operatorname{op}^w(g) = f \ast g$ and $\operatorname{op}^w(f)*g=f*\operatorname{op}^w(g) = \operatorname{op}^w(f*g)$.

In the following, we will make use of the fact that $K = \operatorname{op}^w((1-\Delta)^{-d-1}\delta_0) \in \mathcal T^1(\mathcal H)$, where $\Delta$ is the Laplacian on $\mathbb R^{2n}$. While this can certainly be checked concretely, we just give some abstract reason for this: Since the Fourier transform of $(1-\Delta)^{-d-1}\delta_0$ is (up to constants) given by $(1 + |x|^2 + |\xi|^2)^{-d-1} \in L^1(\Xi)$, the results in \cite{Buzano_Toft2010} can be utilized to verify that this is indeed true. 
\begin{remark}
    The precise nature of these arguments is not particularly important. One could also argue that $\delta_0 \in H^{-d-1}(\mathbb R^{2d})$ such that $(1-\Delta)^{-m} \delta_0 \in H^{2d+1}(\mathbb R^{2d})$ for $m$ large enough. By a classical result of Daubechies, it is $\operatorname{op}^w(H^{2d+1}(\mathbb R^{2d})) \subset \mathcal T^1(\mathcal H)$. Cordes, in his original work \cite{Cordes1979}, used instead the fact that
    \begin{align*}
        \operatorname{op}^w((1-\frac{\partial}{\partial x_1})^{-s} \dots (1- \frac{\partial}{\partial x_d})^{-s}(1 - \frac{\partial}{\partial \xi_1})^{-s}(1 - \frac{\partial}{\partial \xi_d})^{-s} \delta_0) \in \mathcal T^1(\mathcal H)
    \end{align*}
    for $s$ sufficiently large, which he verified by hand. Any of these arguments are sufficient to prove the smoothness result that we are aiming for. Indeed, any of these arguments yield different forms of the Calder\'{o}n-Vaillancourt and reverse Calder\'{o}n-Vaillancourt theorems that we discuss below. Since none of these estimates yields particularly sharp results, the outcome is not particularly important: Sharper results would use fractional powers of the differential operators, so we would have to discuss fractional powers of phase space differentials of operators - a topic which we do not want to touch here.
\end{remark}

For $f\in \S'(\RR^{2d})$ and $P$ the differential operator $P = (1-\Delta)^d$, it follows from the basic properties of quantum harmonic analysis that
\begin{align}
    \opw(f) = \opw(f) \ast \delta_0 = f * \opw(\delta_0) = P(f) * K.
\end{align}
Applying $\|P(f)\|_{L^p}\le \|f\|_{W^{2d,p}}$ and the QHA version of Young's convolution inequality proves a version of the Calder\'on-Vaillancourt estimates for Schatten-classes:
\begin{theorem}
    Let $p\in [1,\infty)$ and let $f\in W^{2d,p}(\RR^{2d})$. Set $c=(2\pi)^{d/p}\|K\|_{\mc T^1}$, then 
    \begin{equation}
        \|\opw(f)\|_{\mc T^p} \le c\,\|f\|_{W^{k,p}}.
    \end{equation}
\end{theorem}

A result of this type was proved first by Arsu in \cite{Arsu2008} by generalizing a proof of the Calder\'on-Vaillancourt theorem due to Cordes.
Essentially his proof uses the same techniques that we used, just without explicitly using the language of quantum harmonic analysis. We want to emphasize that a collection of related estimates was recently obtained in \cite{Lafleche2024}.

The above theorem implies that the pseudodifferential operators with symbols in the space $W^{\infty,p}(\RR^{2d})$ satisfy a Schatten-$p$ analogue of Heisenberg-smoothness.
To formalize this and to state a kind of converse inequality, we introduce \emph{quantum Sobolev spaces} $W^{k,p}(\H)$, which are an operator version of the Sobolev spaces on phases.
These are the Banach spaces of operators $A\in \mc T^p(\H)$ such that $z\mapsto \alpha_z(A)$ is a $C^k$ mapping from $\RR^{2d}$ to $\mc T^p(\H)$.
For $\alpha\in\NN_0^{2d}$ we again get a linear contraction \begin{equation}
    \partial^\alpha: W^{k,p}(\H)\to W^{k-|\alpha|,p}(\H)
\end{equation}
which allows us to define the norm on $W^{k,p}(\H)$ just as the Sobolev norms on $W^{k,p}(\RR^{2d})$ are defined: 
\begin{align*}
    \| A\|_{W^{k,p}} = \sum_{|\alpha| \leq k} \| \partial^\alpha A\|_{\mathcal T^p}.
\end{align*}
Further, we let $W^{\infty, p}(\mathcal H) = \cap_{k=0}^\infty W^{k,p}(\mathcal H)$. Again, we refer to \cite{Lafleche2024} for many results related to quantum Sobolev spaces.

We mention the following easy-to-check properties to emphasize how natural these spaces are in the context of quantum harmonic analysis
\begin{lemma} Let $ p \in [1, \infty)$ and $k \in \mathbb N_0$. 
    \begin{enumerate}
        \item $W^{k,p}(\mathcal H)$ is a Banach space.
        \item $W^{\infty, p}(\mathcal H)$ is a Fr\'{e}chet space.
        \item $\mc S(\H)\hookrightarrow W^{k,p}(\H)$ with dense image.
        \item If $m \geq k$, then $W^{k,p}(\mathcal H) \hookrightarrow W^{m,p}(\mathcal H)$ densely.
        \item If $p' \in (p, \infty)$, then $W^{k,p}(\mathcal H) \hookrightarrow W^{k, p'}(\mathcal H)$.
        \item $W^{k,p}(\mathcal H)$ is a left- and right-$C^{k}(\mathcal H)$-module. It is $\| AB\|_{W^{k,p}} \leq \| A\|_{C^{k}} \| B\|_{W^{k,p}}$ and $\| BA\|_{W^{k,p}} \leq \| B\|_{W^{k,p}} \| A\|_{C^{k}}$.
    \end{enumerate}
\end{lemma}

We denote the Weyl symbol of an operator $A\in \mc L(\H)$ by $\operatorname{sym}^w(A)$, i.e.\ $\operatorname{sym}^w(A)\in\mc S'(\RR^{2d})$ is the tempered distribution such that $\operatorname{op}^w(\operatorname{sym}^w(A))=A$. 
For every $A\in W^{2d,p}(\H)$ we have
\begin{equation}
    \symw(A) = A * \opw(\delta_0) = P(A) * K
\end{equation}
to which we again apply the QHA version of Young's convolution inequality to obtain:
\begin{theorem}
    Let $p\in [1,\infty)$ and let $A\in W^{2d,p}(\H)$. Set $c=(2\pi)^{d/p}\|K\|_{\mc T^1}$, then 
    \begin{equation}
        \|\symw(A)\|_{L^p} \le c\,\|A\|_{W^{k,p}}.
    \end{equation}
\end{theorem}

Since the Weyl quantization is covariant and linear, these estimates immediately imply the $p$-version of Cordes' theorem:

\begin{theorem}
    Let $p\in [1,\infty)$. The space $W^{\infty,p}(\H)$ consists precisely of Weyl pseudodifferential operators with symbols in $W^{\infty,p}(\RR^{2d})$.
\end{theorem}

\section{Heisenberg-analytic operators}\label{sec:analytic}

In this last section, we are again only concerned with the phase space $\Xi = \mathbb R^{2n}$ and Weyl quantization. The results in this section can again be modified to obtain analogous statements for other quantizations. 

We recall that a function $f: \mathbb R^{2n} \to X$, where $X$ is some Banach space, is called \emph{real analytic} provided for each $x_0 \in \mathbb R^{2n}$ the function $f$ can be expressed in a neighborhood around $x_0$ as a convergent power series:
\begin{align*}
    f(x) = \sum_{\beta \in \mathbb N_0^{2n}} a_\beta (x-x_0)^\beta, \quad a_\beta \in X, |x-x_0| < \varepsilon.
\end{align*}
Clearly, for a function to be real analytic, it is necessary that the function is smooth. The following fact relating smooth and analytic functions is well-known:
\begin{lemma}\label{lem:real_analytic}
    Let $X$ be a Banach space $f: \mathbb R^{2n} \to X$ be $C^\infty$. Then, $f$ is analytic if and only if for each $x_0 \in \mathbb R^{2n}$ there exists an $\varepsilon > 0$ and $C> 0$, $R > 0$ such that for all $x \in \mathbb R$ with $|x-x_0| < \varepsilon$:
    \begin{align*}
        \| \partial^\beta f(x)\|_X \leq C \frac{\beta!}{R^{|\beta|}}.
    \end{align*}
\end{lemma}
A proof for the scalar case (i.e., $X = \mathbb C$) can be found in \cite[Prop.\ 2.2.10]{Krantz_Parks2002}, and the proof works for Banach space-valued functions with the obvious modifications.

Let $f \in L^\infty(\Xi)$ and $A \in \mathcal L(\mathcal H)$. We say that $f$ is \emph{uniformly analytic} and $A$ is Heisenberg-analytic, respectively, if the functions $h_f: \Xi \to L^\infty(\Xi)$ and $h_A: \Xi \to \mathcal L(\mathcal H)$ given by
\begin{align*}
    h_f(z) = \alpha_z(f), \quad h_A(z) = \alpha_z(A) \in \mathcal L(\mathcal H)
\end{align*}
are real analytic functions. Using the natural covariance properties of these maps, it is not hard to see that these functions are real analytic if and only if they are analytic in a neighborhood of zero, i.e., if for some $\varepsilon > 0$ and $|z| < \varepsilon$ we have:
\begin{align*}
    \alpha_z(f) = \sum_{\beta \in \mathbb N_0^{2n}} a_\beta z^\beta, \quad \alpha_z(A) = \sum_{\beta \in \mathbb N_0^{2n}} b_\beta z^\beta, \quad a_\beta \in L^\infty(\Xi), ~b_\beta \in \mathcal L(\mathcal H). 
\end{align*}
Note that by evaluating derivatives of these expressions at zero, we immediately obtain that $a_\beta \in C_b^\infty(\Xi)$ and $b_\beta \in C^\infty(\mathcal H)$. Based on this and Lemma \ref{lem:real_analytic}, we obtain the following:
\begin{lemma}
    Let $f \in C_b^\infty(\Xi)$ and $A \in C^\infty(\mathcal H)$. Then, $f$ is uniformly analytic if and only if there exists $C, R > 0$ such that for every $\beta \in \mathbb N_0^{2n}$:
    \begin{align}\label{cond_HeisenbergAnalytic}
        \|\partial^\beta f\|_\infty \leq C \frac{\beta!}{R^{|\beta|}}.
    \end{align}
    Analogously, $A$ is Heisenberg-analytic if and only if there exists $C, R > 0$ such that for every $\beta \in \mathbb N_0^{2n}$:
    \begin{align}
        \|\partial^\beta A\|_\infty \leq C \frac{\beta!}{R^{|\beta|}}.
    \end{align}
\end{lemma}
Note that uniform analyticity of $f$ enforces real-analyticity of $f$ in the classical sense, i.e., for every $w \in \mathbb R^{2n}$ there exists a neighborhood in which we have:
\begin{align*}
    f(z) = \sum_{\beta \in \mathbb N_0^{2n}} \frac{\partial^\beta f(w)}{\beta!}(w-z)^\beta.
\end{align*}
In particular, a function $f$ is uniformly analytic if and only if it is real-analytic on all of $\mathbb R^{2n}$ and satisfies \eqref{cond_HeisenbergAnalytic}.

Based on the previous lemma, we arrive at the following result concerning Heisenberg-analytic operators:
\begin{theorem}
Let $A = \operatorname{op}^w(f)$. Then, $A$ is Heisenberg-analytic if and only if $f$ is uniformly analytic.
\end{theorem}
\begin{proof}
    We only prove that $A$ is Heisenberg-analytic whenever $f$ is uniformly analytic, the other implication is obtained by the same argument with the reverse Calder\'{o}n-Vaillancourt estimate. Let $C, R > 0$ such that
    \begin{align*}
        \| \partial^\beta f\|_\infty \leq C \frac{\beta!}{R^{|\beta|}}.
    \end{align*}
    Without loss of generality, we may assume that $R < 1$. Then, we obtain from Calder\'{o}n-Vaillancourt:
    \begin{align*}
        \| \partial^\beta \operatorname{op}^w(f)\|_{op} &\leq C_1 \sum_{|\gamma| \leq 2n+1} \| \partial^{\beta + \gamma}f\|_\infty\\
        &\leq C_1 C\sum_{|\gamma| \leq 2n+1} \frac{(\beta + \gamma)!}{R^{|\beta + \gamma|}}\\
        &= C_1 C \frac{\beta!}{R^{|\beta|}} \sum_{|\gamma| \leq 2n+1} \frac{(\beta + \gamma)!}{\beta ! R^{|\gamma|}}\\
        &\leq C_1 C \frac{\beta!}{R^{|\beta|}} \sum_{|\gamma| \leq 2n+1} \frac{1}{R^{|\gamma|}} \cdot |\beta|^{|\gamma|}
        \intertext{With $C_{2,R} = C_1 C \sum_{|\gamma| \leq 2n+1} R^{-|\gamma|}$, we obtain:}
        &\leq C_{2,R} \frac{\beta!}{R^{|\beta|}} \cdot |\beta|^{2n+1}\\
        &= C_{2,R} \frac{\beta!}{(R^{3/2})^{|\beta|}} \cdot \frac{|\beta|^{2n+1}}{R^{-|\beta|/2}}
        \intertext{Since $|\beta|^{2n+1}$ grows polynomially in $\beta$ and $R^{-|\beta|/2}$ grows exponentially (since $R^{-1} > 1$), we see that $|\beta|^{2n+1}/R^{-|\beta|/2}$ is a bounded function of $\beta$ such that with appropriate $C_{3,R} > 0$ and $S = R^{3/2} > 0$:}
        \| \partial^\beta \operatorname{op}^w(f)\|_{op} &\leq C_{3, R} \frac{\beta!}{S^{|\beta|}}.
    \end{align*}
    By Lemma \ref{lem:real_analytic}, we obtain that $z \mapsto \alpha_z(A)$ is real-analytic in a neighborhood of zero.
\end{proof}
\begin{example}
    Let $f(x, y) = \cos(x) \sin(y)$. Then, $f \in C_b^\infty(\mathbb R^2)$. By Cordes' theorem, $\operatorname{op}^w(f)$ is Heisenberg-smooth. Since $\| \partial^\alpha f\|_\infty = 1 \leq \frac{\alpha!}{1^{|\alpha|}}$, the operator is even Heisenberg-analytic.
\end{example}
To add another observation, we want to mention the following:
\begin{proposition}\label{proof:analyticspectralinvariance}
    The set of uniformly analytic functions is a spectrally invariant dense subalgebra of $C_b^\infty(\Xi)$. The set of Heisenberg analytic operators is a spectrally invariant dense subalgebra of $C_b^\infty(\mathcal H)$. 
\end{proposition}
\begin{proof}
    We only sketch the proof for the operators, the statement for functions can be dealt with analogously. We first deal with the spectral invariance. In principle, these are just standard arguments concerning analytic functions. Since we have some extra structure here (the inverse needs to be of the form $\alpha_z(A^{-1})$) and our analytic functions at hand take values in a noncommutative $C^\ast$-algebra, we go through these arguments for clarity.

    Assume that $A$ is Heisenberg analytic and invertible. Without loss of generality, we may assume that $\| A\| \leq \frac 12$. To fix the notation for the coefficients, we write 
    \begin{align*}
        \alpha_z(A) = \sum_{\alpha \in \mathbb N_0^{2n}} a_\alpha z^\alpha.
    \end{align*}
    It is clear that the 0-th coefficient $a_0$ in this expansion is again $A$ such that this coefficient is invertible. We now use the classical formula for the inverse of formal power series: Letting $b_0 = a_0^{-1}$ and
    \begin{align*}
        b_\beta = -a_0^{-1} \sum_{0 \leq \alpha < \beta} a_\alpha b_{\beta - \alpha}, \quad \beta \in \mathbb N_0^{2n},
    \end{align*}
    we recursively define the coefficients $b_\beta$. From formal computations, it is now clear that the formal power series
    \begin{align*}
        g(z) = \sum_{\beta \in \mathbb N_0^{2n}} b_\beta z^\beta
    \end{align*}
    satisfies $1 = \alpha_z(A)g(z)$, provided the series converges. We need to prove that the power series defining $g$ actually converges in some neighborhood of zero. For doing so, we fix $R > 0$ small enough such that $\sum_{\alpha \in \mathbb N_0^{2n}} \| a_\alpha\| R^{|\alpha|}$ actually converges. Since this sum depends, within its domain of convergence, continuously on $R$ and equals $\frac 12$ for $R = 0$, we may even choose $R$ small enough such that
    \begin{align*}
        \sum_{\alpha \in \mathbb N_0^{2n}} \| a_\alpha\| R^{|\alpha|} \leq 1.
    \end{align*}
    Now, we will prove by induction that $\| b_\beta\| \leq C \frac{\beta!}{(R/\| a_0^{-1}\|)^{|\beta|}}$. For $\beta = 0$, the statement is clear. For $\beta \neq 0$, we observe that:
    \begin{align*}
        \| b_\beta\| &\leq \| a_0^{-1}\| \sum_{0 \leq \alpha < \beta} \| a_\alpha\| \| b_{\beta - \alpha}\| \leq \| a_0^{-1}\| \sum_{0 \leq \alpha < \beta} \| a_\alpha\| \frac{(\beta - \alpha)!}{(R/\| a_0^{-1}|)^{|\beta - \alpha|}}\\
        &\leq C \| a_0^{-1}\| \| a_0^{-1}\|^{|\beta|-1} \frac{\beta!}{R^{|\beta|}} \sum_{0 \leq \alpha <\beta} \| a_\alpha\| R^{|\alpha|} \leq C  \frac{\beta!}{(R/\| a_0^{-1}\|)^{|\beta|}} \sum_{\alpha \in \mathbb N_0^{2n}} \| a_\alpha\| R^{|\alpha|}\\
        &\leq  C  \frac{\beta!}{(R/\| a_0^{-1}\|)^{|\beta|}}.
    \end{align*}
    Hence, $g$ is a real-analytic function. As the last step, we observe that $\alpha_z(A) g(z) = I$ for $z$ in a neighborhood of $0$. Hence, $g(z) = \alpha_z(A)^{-1} = \alpha_z(A^{-1})$ for such $z$. This shows that $A$ has a right-inverse in the class of Heisenberg-analytic operators. Analogously, one can construct a left-inverse, hence $A$ is invertible in this class.

    Now we come to density of the class of Heisenberg-analytic operators. We claim that for any $A \in \mathcal C(\mathcal H)$, $g_t \ast A$ is Heisenberg-analytic, where $g_t(z) = \frac{1}{(\pi t)^n} e^{-\frac{|z|^2}{t}}$. For doing so, we prove that
\begin{align}\label{est:Gaussian}
    \| \partial^\beta g_t\|_{L^1} \leq C \frac{\beta!}{R^{|\beta|}}
\end{align}
for some $C, R > 0$. Then, by the Hausdorff-Young inequality and covariance of the convolution with respect to taking derivatives, it follows that $g_t \ast A$ is Heisenberg-smooth. Since $g_t \ast A \to A$ in operator norm as $t \to 0$, the claim then follows. 

For proving Eq.~\eqref{est:Gaussian}, it suffices to prove the result for $t = 1$ by a standard substitution. Furthermore, by the tensor product structure of Gaussians, it suffices to prove that
\begin{align*}
    \| d^k f\|_{L^1} \leq C \frac{k!}{R^k}
\end{align*}
for some $C, R > 0$, where $f(x) = e^{-x^2},~x \in \mathbb R$. By Rodriguez' formula, 
\begin{align*}
    d^k f(x) = (-1)^k H_k(x) e^{-x^2},
\end{align*}
where $H_k$ is the $k$-th Hermite polynomial. Hence, we need to estimate $\| H_k e^{-(\cdot)^2}\|_{L^1}$. To this end, 
\begin{align*}
    \| H_k e^{-(\cdot)^2}\|_{L^1} &= \int_{\mathbb R} |H_k(x)| e^{-x^2}~dx\\
    &\leq \int_{\mathbb R} |H_k(x)| e^{-\frac{x^2}{2}}~dx\\
    &= \int_{\mathbb R} \left [H_k^2(x) e^{-x^2} \right]^{\frac{1}{2}}~dx\\
    &= \| H_k^2 e^{-(\cdot)^2}\|_{L^{\frac 12}}^{\frac{1}{2}}.
\end{align*}
The asymptotics of this expression are known: From \cite[Theorem 2.1]{Aptekarev_Dehesa_Sanchez2012} we know that for $k \to \infty$:
\begin{align*}
    \| H_k^2 e^{-(\cdot)^2}\|_{L^{\frac 12}}^{\frac{1}{2}} &= c (2(k+1))^{\frac{k+1}{2}}e^{-\frac{k+1}{2}} (1 + o(1))\\
    &\lesssim 2^{\frac{k+1}{2}} \left( \frac{k+1}{e}\right)^{\frac{k+1}{2}} (1+o(1))
    \intertext{By Stirling's approximation, we see that:}
    &= 2^{\frac{k+1}{2}} \left [\frac{(k+1)!}{\sqrt{2\pi(k+1)}}\right ]^{\frac{1}{2}} (1+o(1))\\
    &\cong 2^{\frac{k}{2}} \sqrt{k+1} (k!)^{\frac{1}{2}}(1+o(1))\\
    &\lesssim 2^k k! (1+o(1)).
\end{align*}
Hence, we see that there exists a constant $C > 0$ such that
\begin{align*}
    \| d^k e^{-(\cdot)^2}\|_{L^1} \leq C \frac{k!}{(1/2)^k},
\end{align*}
which finishes the proof.
\end{proof}
We could also say that a function $f: \mathbb R^{2n} \to \mathbb C$ is $L^p$-analytic if $z \mapsto \alpha_z(f)$ can be expressed as a convergent power series in $L^p(\mathbb R^{2n})$, 
\begin{align*}
    \alpha_z(f) = \sum_{\mathbb N^{2n}} a_\beta z^\beta
\end{align*}
for all $|z| < \varepsilon$, where $a_\beta \in L^p(\mathbb R^{2n})$. Analogously, the operator $A$ is Schatten-$p$-analytic if
\begin{align*}
    \alpha_z(A) = \sum_{\beta \in \mathbb N^{2n}} b_\beta z^\beta
\end{align*}
converges for $|z| < \varepsilon$ in $\mathcal T^p(\mathcal H)$. Then, reasoning entirely analogous to the case $p = \infty$ above yields:
\begin{theorem}
    Let $A = \operatorname{op}^w(f)$. Then, $A$ is Schatten-$p$-analytic if and only if $f$ is $L^p$-analytic.
\end{theorem}
\begin{example}
    As part of the proof of Proposition \ref{proof:analyticspectralinvariance}, we have already obtained that a Gaussian $g$ is $L^1$-analytic, hence $\operatorname{op}^w(g)$ is trace class analytic.
\end{example}

\subsection*{Acknowledgements.}
LvL acknowledges funding by the MWK Lower Saxony (Stay Inspired Grant 15-76251-2-Stay-9/22-16583/2022).

\bibliographystyle{abbrv}
\bibliography{main}

\begin{multicols}{2}

\noindent
Robert Fulsche\\
\href{fulsche@math.uni-hannover.de}{\Letter ~fulsche@math.uni-hannover.de}
\\
\noindent
Institut f\"{u}r Analysis\\
Leibniz Universit\"at Hannover\\
Welfengarten 1\\
30167 Hannover\\
GERMANY\\

\noindent
Lauritz van Luijk\\
\href{lauritz.vanluijk@itp.uni-hannover.de}{\Letter ~lauritz.vanluijk@itp.uni-hannover.de}
\\
\noindent
Institut f\"{u}r Theoretische Physik\\
Leibniz Universit\"at Hannover\\
Appelstra\ss e 2\\
30167 Hannover\\
GERMANY

\end{multicols}
\end{document}